\RequirePackage{fix-cm}
\documentclass[smallextended]{svjour3}       
\smartqed  
\usepackage{graphicx}
\usepackage{mathptmx}      
\usepackage{latexsym}
\usepackage{amsmath}
\usepackage{amssymb}
\usepackage{amsfonts}
\usepackage{mathrsfs}

\newcommand{\R}{\mathbb R}
\newcommand{\N}{\mathbb N}

\newcommand{\V}{\mathbb V}

\newcommand{\Uball}{{\mathbb B}}
\newcommand{\Usfer}{{\mathbb S}}

\newcommand{\dom}{{\rm dom}\, }

\newcommand{\nullv}{\mathbf{0}}
\newcommand{\card}{{\rm card}\, }
\newcommand{\krn}{{\rm ker}\, }

\newcommand{\clco}{\overline{\rm co}\, }

\newcommand{\conv}{{\rm co}\, }

\newcommand{\bd}{{\rm bd}\, }
\newcommand{\inte}{{\rm int}\, }

\newcommand{\Lin}{{\bf M}_{m\times n}(\R)}

\newcommand{\SFP}{{\rm SFP}}
\newcommand{\RSFP}{{\rm RSFP}}
\newcommand{\SCQ}{{\rm SCQ}}

\newcommand{\Amap}{\mathbf{A}}
\newcommand{\SVMAO}{\mathcal{A}_\mathcal{U}}
\newcommand{\Solv}{{\rm S}_{C,\mathcal{U},Q}}
\newcommand{\Bpoly}{\mathbf{B}}

\newcommand{\pmfun}{{\nu}_{\mathcal{A}_\mathcal{U}}}     
\newcommand{\val}{{\rm v}}
\newcommand{\ebcon}{\tau}
\newcommand{\ebconglo}{{\tau}_{\mathcal{U}}}

\newcommand{\dcone}[1]{{#1}^{{}^\ominus}}

\newcommand{\sur}[1]{{\rm cov}(#1)}

\newcommand{\ball}[2]{{\rm B}\left[#1, #2\right]}

\newcommand{\dist}[2]{{\rm dist}\left(#1;#2\right)}
\newcommand{\exc}[2]{{\rm exc}(#1;#2)}
\newcommand{\supp}[2]{\varsigma\left(#1;#2\right)}

\newcommand{\Ncone}[2]{{\rm N}(#1;#2)}
\newcommand{\haus}[2]{{\rm haus}(#1;#2)}

\newcommand{\inc}[3]{{\rm inc}(#1;#2;#3)}   

\begin{document}

\title{On a robust approach to ``split" feasibility problems: solvability
and global error bound conditions
}

\titlerunning{Solvability and error bound conditions for robust
split feasibility problems}        

\author{Amos Uderzo}

\authorrunning{A. Uderzo} 

\institute{A. Uderzo \at
              Department of Mathematics and Applications \\
              Università degli studi di Milano-Bicocca  \\
              \email{amos.uderzo@unimib.it}           
}

\date{Received: date / Accepted: date}

\maketitle

\noindent Date: \today

\begin{abstract}
In the present paper, a robust approach to a special class of
convex feasibility problems is considered. By techniques of convex and
variational analysis, conditions for the existence of robust feasible solutions
and related error bounds are investigated. This is done by reformulating
the robust counterpart of a split feasibility problem as a set-valued
inclusion, a problem for which one can take profit from the solvability
and stability theory that has been recently developed.
As a result, a sufficient condition for solution existence and error bounds
is established in terms of problem data and discussed through several
examples. A specific focus is devoted to error bound conditions in the case
of the robust counterpart of polyhedral split feasibility problems.

\end{abstract}


\section{Introduction}
\label{Sect:1}

In the present paper, after \cite{CenElf94} by ``split"
feasibility problem the following mathematical question is meant: given a matrix $A\in\Lin$
and two closed convex sets $C\subseteq\R^n$ and $Q\subseteq\R^m$,
$$
   \hbox{find $x\in C\quad$ such that}
   \quad Ax\in Q. \leqno (\SFP)
$$
This kind of problems was firstly considered in \cite[Section 6.1]{CenElf94}
in the context of iterative projection methods for solving more general convex feasibility
problems, the latter meaning to find a point in the nonempty intersection of finitely many
closed convex sets (see, for instance, \cite{BauBor93,BauBor96}).
As it was recognized soon, the format $(\SFP)$ is relevant
from the viewpoint of both theory and applications. It allows to
formalize classic constraint systems in mathematical programming
and in affine complementarity problems.
Besides, $(\SFP)$ emerges as a model
of interest in applications to engineering, in particular to the area of
signal processing, such as phase retrieval, image reconstruction and computed tomography
issues (see, for instance, \cite{Hurt89}) and to inverse problems
(see \cite{CeElKoBo05}).
As a consequence, a good deal of investigations has been carried out
on various topics related to $(\SFP)$ (see, among the others,
\cite{Byrn02,HuXuYe25,HuXuYe25b,QuLiu18,Xu10}).
In the present paper, by elaborating on the format $(\SFP)$, the
case of problems affected by data uncertainty is addressed.
In handling real-world models leading to $(\SFP)$, there are
several reasons that motivate such an analysis perspective.
To summarize, they include the impossibility to neglect effects
of errors, which unavoidably enter any process of measuring/estimating
parameters that describe a real-world system. Another type of error
also occurring stems from the impossibility to implement a solution with
the same precision degree to which it is computed, making nominal
solution inadequate as a matter of fact to any practical purpose.
In the need to address the phenomenon of data uncertainty, several
methodologies have been developed in optimization.
The investigations exposed in this paper follow the approach proposed
in \cite{BeGhNe09,BenNem98,KouYu97,Soys73}, which since the late
90s is attracting the interest of many researchers
(see, among the others, \cite{BeBoNe06,BenNem02,ElOuLe98,KhTaZa15}
and references therein).
Roughly speaking, the philosophy behind this approach consists in
regarding constraint systems as ``hard" (violations of constraints
can not be admitted at all by the decision maker), while uncertainty is
handled without making any reference to its stochastic nature
(the probability distribution of random data could be unknown or
only partially known, the constraint fulfilment in probability
could not suit the aforementioned ``hardness").
According to such an attitude, it must be considered as a feasible solution
only those solutions, which are able to satisfy the constraint
system, whatever the realizations of uncertain parameters are.
This leads to a sort of  ``immunization against uncertainty", which
appears to be suitable in decision-making environments, where a
conservative attitude towards the constraint fulfilment is
required (interested readers can find some of them discussed in detail in
\cite{BenNem98,BeGhNe09}).

In what follows, a set $\Omega$ is assumed to collect
all possible different scenarios (or states of the world),
whose multiplicity causes the uncertainty phenomenon affecting the problem
data. With reference to split feasibility problems, the present
analysis focuses on the uncertainty of the entries of $A$, disregarding
at present the possible uncertainty related to $C$ and $Q$. The realization of the
scenario $\omega$ in terms of entries of $A$ is denoted by $A_\omega$.
The uncertainty set collecting all possible values of the matrix data
consequently becomes $\mathcal{U}=\{A_\omega\ :\ \omega\in\Omega\}$. Thus, the
resulting robust counterpart to $(\SFP)$ amounts to
$$
    \hbox{find $x\in C\quad$ such that}\quad
    A_\omega x\in Q,\quad\forall
   \omega\in\Omega. \leqno (\RSFP)
$$
The aim of the present paper is to conduct a study of solvability and
error bounds for $(\RSFP)$ relying on basic tools from convex and variational
analysis. In order to perform a related technique of analysis,
let us introduce the set-valued map $\SVMAO:\R^n\rightrightarrows\R^m$,
defined by
\begin{equation}   \label{eq:defSVMAO}
  \SVMAO(x)=  \{y\in\R^m\ :\ y=A_\omega x,\ A_\omega\in\mathcal{U}\}.
\end{equation}
Through the above map, $(\RSFP)$ can be reformulated as the following
set-valued inclusion problem
\begin{equation}    \label{in:SVIref}
    \hbox{find $x\in C\quad$ such that}\quad
    \SVMAO(x)\subseteq Q.
\end{equation}
This kind of generalized (set-inclusive) equations was addressed, independently of the
robust optimization approach, already in \cite{Uder19}, but its
appearance in optimization contexts can be traced back at least to \cite{Soys73}.
In the above setting, the set of all solutions of $(\RSFP)$
(henceforth {\it robust feasible solutions}) is given by
\begin{equation}    \label{eq:Solvdef}
  \Solv=\{x\in C \ :\ \SVMAO(x)\subseteq Q\}=C\cap \SVMAO^{+1}(Q)
  =C\cap\bigcap_{\omega\in\Omega}A_\omega^{-1}(Q),
\end{equation}
where the notation $\mathcal{F}^{+1}(S)$ indicates the core of $S$
through $\mathcal{F}$, i.e. one the two
possible ways, in which the inverse image through a set-valued map $\mathcal{F}$
of a given subset $S$ of the range space is meant.

By global error bound for $(\RSFP)$ or (\ref{in:SVIref}), any
situation is meant in which there exists $\tau>0$ such that
$$
  \dist{x}{\Solv}\le\tau \rho(x),\quad\forall x\in\R^n,
$$
where $\rho$ denotes a residual function, i.e.
any function $\rho:\R^n\longrightarrow [0,+\infty]$ working in
such a way that
$$
  x\in\Solv \qquad\hbox{ iff }\qquad
  \rho(x)=0,
$$
whose values can be computed by problem data.
The study of error bounds for various kinds of problems
(e.g. scalar inequality systems and cone constraints,
variational inequalities and complementarity problems, equilibrium problems)
is a topic widely explored in the last decades, because of
connections with profound themes in variational analysis
(regularity phenomena, implicit functions, and
Lipschitz behaviours of maps) and the broad range of applications
in different areas, such as penalty function methods in mathematical
programming, implementation and convergence analysis of numerical
methods for solving optimization problems. As a result, the existing
literature on the subject is vast (in the impossibility to provide
an exhaustive account, \cite{FaHeKrOu10,FacPan03,Pang97,WuYe02}
may be indicated as a source for a bibliographical search).
In spite of this, to the best of his knowledge, the author has no information
about any attempt to study error bounds specific for $(\RSFP)$
in the recorded literature.
Local error bounds for nonlinear split feasibility problems have
been investigated in \cite{GaShWa23}.
A stability/sensitivity analysis of split feasibility problems has been recently
conducted by variational analysis techniques, in various settings,
in \cite{HuXuYe25,HuXuYe25b}. Nevertheless,
as pointed out in \cite{BeGhNe09}, the methodology for handling data perturbation
proposed by stability/sensitivity analysis ``is aimed at completely different questions".
Indeed, instead of being concerned with the effects of the uncertainty set as a whole
on the very feasible solution concept (robust feasibility), the analysis in that framework deals
with reactions to (often small) perturbations near a reference value of
the classic solution set. The reader should notice that this essential difference
reflects on the different use of set-valued maps arising in the two methodologies.

The contents of the paper are arranged according to the following scheme.
In Section \ref{Sect:2} various preliminary elements (discussions about
the standing assumptions, needed tools from convex and variational analysis,
the main technique of investigation) are presented.
Section \ref{Sect:3} exposes the main result of the paper. In the same section,
the reader will find several examples and comments aimed at illustrating
some features of such a result.
In Section \ref{Sect:4} a specific focus on error bound conditions for
the robust counterparts of split
feasibility problems in the polyhedral case is discussed.
A section for conclusions to be drawn completes the contents.

The basic notations used throughout the paper are standard.
$\R$ denotes the field of real numbers, while $\N$ denotes the subset of
natural numbers. $\R^n$ stands for the space of
vectors with $n$ real components, $\R^n_+$ the nonnegative orthant
in $\R^n$, with $\nullv$ denoting the null vector.
In any Euclidean space, $\|\cdot\|$ stands for the Euclidean norm.
Given a function $\varphi:\R^n\longrightarrow\R\cup\{\pm\infty\}$,
$\dom\varphi=\varphi^{-1}(\R)$ denotes its domain and, if $\ell\in\R$,
$[\varphi\le\ell]=\{x\in\R^n\ :\ \varphi(x)\le\ell\}$ denotes its
sublevel set, whereas $[\varphi>\ell]=\{x\in\R^n:\ \varphi(x)>\ell\}$
denotes its strict superlevel set. If $S$ is a subset of $\R^n$,
$\dist{x}{S}=\inf_{z\in S}\|z-x\|$ denotes that distance of a point $x\in\R^n$ from
$S$, with the convention that $\dist{x}{\varnothing}=+\infty$.
Consistently, if $r\ge 0$, $\ball{S}{r}=[\dist{\cdot}{S}\le r]$
indicates the $r$-enlargement of the set $S$, with radius $r$. In particular, if
$S=\{x\}$, $\ball{x}{r}$ denotes the closed ball with center $x$
and radius $r$.
In this context, $\ball{\nullv}{1}$ and $\bd\ball{\nullv}{1}$
will be simply indicated by $\Uball$ and $\Usfer$, respectively.
The symbol $\bd S$, $\inte S$ and $\clco S$ indicate the boundary, the topological interior
and the convex closure of $S$, respectively. Whenever $S$ is a vector
subspace of an Euclidean space, by $S^\perp$ its orthogonal complement
is denoted.
Given a pair of subsets, say $A$ and $S$, the excess of $A$ over $S$
is indicated by $\exc{A}{S}=\sup_{a\in A}\dist{a}{S}$, whereas their
Hausdorff-Pompeiu distance by $\haus{A}{S}=\max\{\exc{A}{S},\, \exc{S}{A}\}$.
Given a set-valued map $\mathcal{F}:\R^n\rightrightarrows\R^m$, $\dom\mathcal{F}=
\{x\in\R^n\ :\ \mathcal{F}(x)\ne\varnothing\}$ denotes the effective domain
of $\mathcal{F}$.
The space $\Lin$ will be equipped with the operator norm, indicated by $\|\cdot\|_{\bf M}$.
Given $A\in\Lin$, $A^\top$ stands for the transposed matrix of $A$
and $\ker A$ stands for its kernel.

The acronyms l.s.c. and p.h., standing for lower semicontinuous
and positively homogeneous, respectively, will be also used.
The meaning of further symbols employed in subsequent sections
will be explained contextually to their introduction.


\section{Preliminaries}
\label{Sect:2}

\subsection{\bf Standing assumptions on the problem data}
In this section, the assumptions standing throughout the
paper are formulated and discussed.

$\bullet$ Assumptions on the (certain) data $C$ and $Q$:
\begin{itemize}
  \item[$(a_0)$] $\varnothing\ne C\subseteq\R^n$ and $\varnothing\ne Q\subseteq\R^m$
  are both closed and convex sets.
\end{itemize}

$\bullet$ Assumptions on uncertain data:
In the main result of the paper, for enabling to make use of certain convex
analysis constructions, $\Omega$ will be supposed to be a separated
compact topological space. Given an uncertainty set
$\mathcal{U}=\{A_\omega\ :\ \omega\in\Omega\}$, it is convenient
to introduce the map $\Amap:\Omega\longrightarrow (\R^m)^{\R^n}$ by setting
$$
  \Amap(\omega)=A_\omega.
$$
The following assumptions are made on $\mathcal{U}$ via the map $\Amap$:
\begin{itemize}
  \item[$(a_1)$] $\Amap(\omega)\in\Lin$, for every $\omega\in\Omega$;
  \\
  \item[$(a_2)$] the set $\Amap(\Omega)=\mathcal{U}$ is a compact and
  convex subset of $\Lin$.
\end{itemize}

Assumption $(a_1)$ means that, even in the presence of uncertainty,
the robust counterpart of the split feasibility problem maintains a linear structure
with respect to the decision variable $x$.

Assumption $(a_2)$ indicates that no specific structure is imposed
on the dependence of $A_\omega$ on $\omega$ (the decision maker knows
only that the scenario $\omega\in\Omega$ yields the ``revealed matrix" $A_\omega$,
not how it does), apart from the aforementioned properties of $\mathcal{U}$.
In other terms, unlike $\mathcal{U}$, the map $\Amap$ may be not among
the problem data. Nonetheless, in dealing with concrete instances
of $(\RSFP)$ and examples to be worked out, the explicit form of $\Amap$ will be given.

\begin{remark}
Notice that assumptions $(a_1)$ and $(a_2)$ are satisfied whenever $\Omega$
is a compact convex subset of a topological vector space $\V$ over the
field $\R$ and the map $\Amap:\V\longrightarrow\Lin$ is linear (or affine) and continuous.
A specific example of such a circumstance is discussed below in details.
\end{remark}

\begin{example}[Birkhoff polytope]    \label{ex:Bpoly}
Let $n=m=3$ and let $\V=\R^4$ be equipped with its usual Euclidean space
structure. Consider the map $\Amap:\R^4\longrightarrow{\bf M}_{3\times 3}(\R)$
defined by
$$
  \Amap(\omega)=\left(\begin{array}{ccc}
                        \omega_1 & \omega_2 & 1-\omega_1-\omega_2 \\
                        \omega_3 & \omega_4 & 1-\omega_3-\omega_4 \\
                        1-\omega_1-\omega_3\quad & 1-\omega_2-\omega_4\quad & \omega_1+\omega_2+\omega_3+\omega_4-1
                      \end{array}    \right),
$$
where $\omega=(\omega_1,\omega_2,\omega_3,\omega_4)\in\R^4$.
In this setting, let us define $\Omega$ to be the following compact polyhedron
(polytope)
$$
  \Omega=\left\{\omega=(\omega_1,\omega_2,\omega_3,\omega_4)\in\R^4_+\ :\
  \begin{array}{l}
                        \omega_1 + \omega_2 \le 1 \\
                        \omega_1 + \omega_3 \le 1 \\
                        \omega_2 + \omega_4 \le 1 \\
                        \omega_3 + \omega_4 \le 1 \\
                        \omega_1+\omega_2+\omega_3+\omega_4\ge 1
                      \end{array}  \right\}\subseteq [0,1]^4.
$$
The map $\Amap$ is clearly affine and therefore $\Amap(\Omega)$ is
a compact convex subset (actually a polytope) of ${\bf M}_{3\times 3}(\R)$.
It is possible to see, indeed, that by construction of $\Amap$
and $\Omega$, the set $\Amap(\Omega)$ coincides with Birkhoff polytope
\begin{eqnarray*}
  \Bpoly_3=\{A=(a_{i,j})_{i,j=1,2,3}\in{\bf M}_{3\times 3}(\R)\ :\
                       & & \sum_{i=1}^{3}a_{i,j} =1,\quad\forall j=1,2,3, \\
                        & & \sum_{j=1}^{3}a_{i,j} =1,\quad\forall i=1,2,3, \\
                        & & a_{i,j}\ge 0,\quad\forall i,j=1,2,3 \}
\end{eqnarray*}
consisting of all the $3$-dimensional doubly stochastic matrices
(see, for instance, \cite[Chapter II.5]{Barv02}). Clearly, the above
example can be generalized to an arbitrary dimension $n\in\N\backslash\{0,1\}$, by a proper
choice of the polytope $\Omega$ in the space $\R^{(n-1)^2}$, in such a
way that $\Amap(\Omega)=\Bpoly_n$.
\end{example}

Of course, assumptions $(a_1)$ and $(a_2)$ may happen to be satisfied even though
$\Omega$ is not convex neither compact, with $\Amap:\Omega\longrightarrow\Lin$
being not affine, as illustrated in the next example.

\begin{example}
Let $n=m=1$, let $\Omega=(-\infty,-1]\cup\{0\}\cup [1,+\infty)\subseteq\R$ and
let $\Amap:\R\longrightarrow{\bf M}_{1\times 1}(\R)\cong\R$ be defined by
$$
   \Amap(\omega)=\frac{\omega}{\omega^2+1}.
$$
In such an event, one has $\Amap(\Omega)=[-1/2,1/2]$, which is compact
and convex.
\end{example}


\subsection{\bf Convex analysis tools}
In this section, some basic concepts from convex analysis are recalled
in view of their subsequent employment. These mainly deal with the well-known
calculus for convex sets and functions.
Given a nonempty convex set $C\subseteq\R^n$ and $\bar x\in C$, denote by
$$
   \Ncone{\bar x}{C}=\{v\in\R^n\ :\ \langle v,x-\bar x\rangle\le 0,
   \quad\forall x\in C\}
$$
the normal cone to $C$ at $\bar x$. By the very definition, it is clear that
if $C$ is in particular a closed cone, then one has $\Ncone{\bar x}{C}\subseteq
\Ncone{\nullv}{C}=\dcone{C}$, where $\dcone{C}$ stands for the negative
dual cone to $C$.

Given a convex function $\varphi:\R^n\longrightarrow\R\cup\{+\infty\}$
and $\bar x\in\dom\varphi$, the symbol
$$
   \partial\varphi(\bar x)=\{v\in\R^n\ :\ \langle v,x-\bar x\rangle\le\varphi(x)
   -\varphi(\bar x),\quad\forall x\in\R^n\}
$$
denotes the Moreau-Rockafellar subdifferential of $\varphi$ at $\bar x$.
In particular, for a convex function $\varphi:\R^n\longrightarrow\R$
it is $\dom\varphi=\R^n$,
so, since a convex function defined on a
finite-dimensional Euclidean space is known to be continuous
on the interior of its domain, then $\partial\varphi(x)\ne\varnothing$
for every $x\in\R^n$ (see \cite[Theorem 3.39(b)]{MorNam22}),
namely $\dom\partial\varphi=\R^n$.
The following calculus rules will be employed in the sequel:

\begin{itemize}

\item[$\bullet$] sum rule: if $\varphi:\R^n\longrightarrow\R$ and
$\vartheta:\R^n\longrightarrow\R$ are convex functions and
$\bar x\in\R^n$, then the following equality holds
\begin{equation}    \label{eq:subdsumr}
   \partial\left(\varphi+\vartheta\right)(\bar x)=
   \partial\varphi(\bar x)+\partial\vartheta(\bar x)
\end{equation}
(see \cite[Theorem 3.48]{MorNam22}).

\item[$\bullet$] chain rule: if $\Lambda\in\Lin$, $\varphi:\R^m\longrightarrow\R$
is a convex function, and $\bar x\in\R^n$ , then it holds
\begin{equation}   \label{eq:subdchainr}
  \partial(\varphi\circ\Lambda)(\bar x)=\Lambda^\top\left(\partial
  \varphi(\Lambda\bar x)\right)
\end{equation}
(see \cite[Theorem 3.55]{MorNam22}).

\item[$\bullet$] Ioffe-Tikhomirov rule: let $(\Xi,\tau)$ be a separated
compact topological space and let $\varphi_\xi:\R^n\longrightarrow\R\cup\{+\infty\}$
be a convex function for every $\xi\in\Xi$. Define $\varphi:\longrightarrow\R
\cup\{+\infty\}$ as being
$$
  \varphi(x)=\sup_{\xi\in\Xi}\varphi_{\xi}(x).
$$
If the function $\varphi_{\xi}(x):\Xi\longrightarrow\R\cup\{+\infty\}$, given by
$\xi\mapsto \varphi_\xi(x)$, is u.s.c. on $\Xi$, $\bar x\in\dom f$ and $\varphi_\xi$
is continuous at $\bar x$ for every $\xi\in\Xi$, then it holds
\begin{equation}    \label{eq:subdITr}
  \partial\varphi(\bar x)=\clco\left(\bigcup_{\xi\in\Xi_{\bar x}}
  \partial\varphi_\xi(\bar x)\right),
\end{equation}
where $\Xi_{\bar x}=\{\xi\in\Xi\ :\ \varphi_\xi(\bar x)=\varphi(\bar x)\}$
indicates the clean-up index set (see \cite[Theorem 2.4.18]{Zali02}).
\end{itemize}

Let $C\subseteq\R^n$ be a nonempty, closed, convex set and $\bar x
\in\R^n$.  In this case, the function $x\mapsto\dist{x}{C}$ turns out to be
convex, with $\dom\dist{\cdot}{C}=\R^n$. Since this function is going to
play a certain role in the analysis of the topic under study, it is convenient
to recall the form taken by its subdifferential, namely
\begin{equation}   \label{eq:subddist}
  \partial\dist{\cdot}{C}(\bar x)=\left\{
  \begin{array}{ll}
    \Ncone{\bar x}{C}\cap\Uball & \quad\hbox{ if } \bar x\in C, \\
    \\
    \left\{\displaystyle{\bar x-\Pi_C(\bar x)\over\dist{\bar x}{C}}\right\}
    & \quad\hbox{ if } \bar x\not\in C,
  \end{array}\right.
\end{equation}
where $\Pi_C:\R^n\longrightarrow C$ denotes the standard Euclidean projector
onto $C$, i.e. $\Pi_C(x)$ stands for the unique element of $C$ such that
$\|x-\Pi_C(x)\|=\dist{x}{C}$ (see \cite[Corollary 3.79]{MorNam22}).
In both the cases it is $\partial\dist{\cdot}{C}(\bar x)\subseteq
\Ncone{\bar x}{C}\cap\Uball$.


\subsection{\bf A subdifferential condition for solvability and error bounds
for convex inequalities}
The basic technique employed to ensure solution existence for $(\RSFP)$
and, at the same time, to establish error bounds relies on the following
proposition, whose proof is provided for the sake of completeness.

\begin{proposition}    \label{pro:sovebconvex}
Let $\varphi:\R^n\longrightarrow\R$ be a convex function. Assume that
$[\varphi>0]\ne\varnothing$ and that
\begin{equation}    \label{def:ebsolvconvex}
      \tau=\inf_{x\in [\varphi>0]}\dist{\nullv}{\partial\varphi(x)}
  =\inf\{\|v\|\ :\ v\in\partial\varphi(x),\quad x\in [\varphi>0]\}.
\end{equation}
Then, it is $[\varphi\le 0]\ne\varnothing$ and
$$
  \dist{x}{[\varphi\le 0]}\le {\varphi(x)\over\tau},
  \quad\forall x\in [\varphi>0].
$$
\end{proposition}

\begin{proof}
Notice that, since it is $\dom\varphi=\R^n$, $\varphi$
is continuous on $\R^n$, which is clearly a complete Banach space.
Thus, both the assertions in the thesis can be established
by applying \cite[Theorem 2.8]{AzeCor06} and by expressing
the strong slope of $\varphi$ in term of its subdifferential,
according to the dual representation established in
\cite[Theorem 5(i)]{FaHeKrOu10}.
\hfill $\square$
\end{proof}

The above and similar, even more general, results are standard and widely
employed tools in variational analysis, which ultimately rest upon
the Ekeland/Bishop-Phelps variational principle or its equivalent reformulations.


\subsection{\bf Functional characterization of $\Solv$}

Throughout this paper, following \cite{Uder19}, in order
to express error bounds for $(\RSFP)$ the residual function
$\pmfun:\R^n\longrightarrow [0,+\infty]$, defined by
\begin{equation}\label{eq:pmfundef}
  \pmfun(x)=\exc{\SVMAO(x)}{Q}+\dist{x}{C}
\end{equation}
will be employed.

The next proposition shows that $\pmfun$ can actually play the role
of a residual function for $(\RSFP)$.

\begin{proposition}      \label{pro:Solvchar}
Given a $(\RSFP)$, under the standing assumptions $(a_0)-(a_2)$, the
following equalities hold:
\begin{equation}    \label{eq:pfunchar}
  \Solv=\pmfun^{-1}(0)=[\pmfun\le 0].
\end{equation}
\end{proposition}

\begin{proof}
As for (\ref{eq:pfunchar}),
if $x\in\Solv$, then one has $\exc{\SVMAO(x)}{Q}=\dist{x}{C}=0$,
so it is true that $x\in [\pmfun\le 0]$.
Vice versa, if $x\in [\pmfun\le 0]$, as $\exc{\SVMAO(x)}{Q}$ and $\dist{x}{C}$
are both nonnegative, then it must
be $\exc{\SVMAO(x)}{Q}=\dist{x}{C}=0$. Since $C$ and $Q$ are closed sets,
these equalities imply $\SVMAO(x)\subseteq Q$ and $x\in\ C$, which means
$x\in\Solv$.
\hfill$\square$
\end{proof}

In order to enhance the result presented in Section \ref{Sect:3},
it is convenient to provide an equivalent reformulation of the
above residual function, which is based on the support function
associated to a given convex set, i.e. given $D\subseteq\R^m$,
the function $\supp{\cdot}{D}:\R^m\longrightarrow\R\cup\{\pm\infty\}$
defined by
$$
  \supp{v}{D}=\sup_{y\in D}\langle y,v\rangle.
$$

\begin{proposition}    \label{pro:dualreppmfun}
Let a problem $(\RSFP)$ satisfy the assumptions $(a_0)-(a_2)$.
Then, it holds
$$
  \pmfun(x)=\sup_{u\in\Uball}[\supp{u}{\SVMAO(x)}-\supp{u}{Q}]+
  \dist{x}{C},\quad\forall x\in\R^n.
$$
\end{proposition}

\begin{proof}
Fix an arbitrary $x\in\R^n$. Clearly, on account of definition
(\ref{eq:pmfundef}), it suffices to show that
$$
  \exc{\SVMAO(x)}{Q}=\sup_{u\in\Uball}[\supp{u}{\SVMAO(x)}-\supp{u}{Q}].
$$
According to the definition of excess, one has
$$
  \exc{\SVMAO(x)}{Q} = \sup_{y\in\SVMAO(x)}\min_{q\in Q}\|y-q\|
   = \sup_{y\in\SVMAO(x)}\min_{q\in Q}\max_{u\in\Uball}
   \langle u,y-q\rangle.
$$
Since $\Uball$ is compact and the inner product is bilinear and continuous, it is
possible to apply the Sion's minimax theorem (see \cite{Sion58}),
in such  a way to obtain
\begin{eqnarray*}
  \exc{\SVMAO(x)}{Q} &=& \sup_{y\in\SVMAO(x)} \max_{u\in\Uball}\min_{q\in Q}
  \langle u,y-q\rangle=
   \sup_{y\in\SVMAO(x)}\max_{u\in\Uball}\left[\langle u,y\rangle-\supp{u}{Q}\right] \\
   &=& \sup_{u\in\Uball}[\supp{u}{\SVMAO(x)}-\supp{u}{Q}],
\end{eqnarray*}
thereby completing the proof. \hfill$\square$
\end{proof}

The present analysis is devised in such a way to take advantage from specific properties of
$\pmfun$, which are inherited by the particular form taken by the set-valued
map $\SVMAO$. To this purpose,
recall that a set-valued map $\mathcal{F}:\R^n\rightrightarrows\R^m$
is said to be
\begin{itemize}
\item[$\bullet$] p.h. on a cone $D\subseteq\R^n$ if
$$
  \mathcal{F}(\lambda x)=\lambda \mathcal{F}(x),\quad
  \forall \lambda>0,\ \forall x\in D;
$$
\item[$\bullet$] concave on the convex set $D\subseteq\R^n$ if
$$
  \mathcal{F}(tx_1+(1-t)x_2)\subseteq  t\mathcal{F}(x_1)+
  (1-t)\mathcal{F}(x_2),\quad
  \forall x_1,\ x_2\in D,\ \forall t\in[0,1];
$$
\item[$\bullet$] Lipschitz continuous on a set $D\subseteq\R^n$ if there exists $l\ge 0$
such that
$$
  \haus{\mathcal{F}(x_1)}{\mathcal{F}(x_2)}\le l\|x_1-x_2\|,
  \quad\forall x_1,\ x_2\in D.
$$
\end{itemize}

\begin{lemma}[Properties of $\SVMAO$]    \label{lem:SVMAOprop}
Let a set-valued map $\SVMAO:\R^n\rightrightarrows\R^m$ be defined
as in (\ref{eq:defSVMAO}). Under the standing assumptions $(a_0)-(a_2)$
the following assertions hold:
\begin{itemize}
  \item[(i)] $\dom\SVMAO=\R^n$;

  \item[(ii)] $\SVMAO$ takes convex and compact values;

  \item[(iii)] $\SVMAO(\nullv)=\{\nullv\}$;

  \item[(iv)] $\SVMAO$ is p.h. and concave on $\R^n$;

  \item[(v)] $\SVMAO$ is Lipschitz continuous on $\R^n$.

\end{itemize}
\end{lemma}

\begin{proof}
(i) The equality trivially follows from formula (\ref{eq:defSVMAO})
and the nonemptiness of $\Omega$.

(ii) Fix an arbitrary $x\in\R^n$ and consider the evaluation map $\val_x:\Lin
\longrightarrow\R^m$, which is defined as
$$
   \val_x(A_\omega)=A_\omega x.
$$
Observe that, owing to $(a_1)$,  $\val_x$ is a linear map acting between finite-dimensional Euclidean spaces
and, as such, it is continuous. Since it is $\SVMAO(x)=\val_x(\Amap(\Omega))$
and, according to assumption $(a_2)$, $\Amap(\Omega)$ is compact, then also
its image through $\val_x$ is compact.
Since, according to assumption $(a_2)$, $\Amap(\Omega)$ is convex, also
$\val_x(\Amap(\Omega))$ is convex.

(iii) By assumption $(a_1)$ one has $A_\omega\nullv=\nullv$ for every
$\omega\in\Omega$, whence the equality in the assertion.

(iv) Both the mentioned properties are straightforward consequences of
the linearity of each map $A_\omega$.

(v) As $\Amap_\omega$ is a compact subset of $\Lin$, it must
be bounded, so there exists $\beta>0$ such that
\begin{equation}  \label{in:lipSVMAO}
 \|A_\omega\|_{\bf M}\le\beta,\quad\forall\omega\in\Omega.
\end{equation}
Take arbitrary $x_1,\, x_2\in\R^n$ and $y\in\SVMAO(x_1)$. Then,
it must be $y=A_{\omega_1}x_1$ for some $\omega_1\in\Omega$.
Then, by linearity of $A_{\omega_1}$ from inequality (\ref{in:lipSVMAO})
one obtains
\begin{eqnarray*}
  \dist{y}{\SVMAO(x_2)} &=& \dist{A_{\omega_1}x_1}{\SVMAO(x_2)}
    =\inf_{\omega\in\Omega} \|A_{\omega_1}x_1-A_\omega x_2\| \\
   &\le& \|A_{\omega_1}x_1-A_{\omega_1}x_2\|\le \|A_{\omega_1}\|_{\bf M}
   \|x_1-x_2\|  \\
   &\le &\beta\|x_1-x_2\|.
\end{eqnarray*}
Thus, by arbitrariness of $y\in\SVMAO(x_1)$ it is possible to deduce
\begin{eqnarray*}
  \exc{\SVMAO(x_1)}{\SVMAO(x_2)} &=& \sup_{y\in\SVMAO(x_1)}\dist{y}{\SVMAO(x_2)}  \\
   &\le &\beta\|x_1-x_2\|,\quad\forall x_1,\, x_2\in\R^n.
\end{eqnarray*}
Interchanging the role of $x_1,\, x_2\in\R^n$ leads to
$$
    \haus{\SVMAO(x_1)}{\SVMAO(x_2)}\le\beta\|x_1-x_2\|,\quad\forall x_1,\, x_2\in\R^n.
$$
This completes the proof.
\hfill $\square$
\end{proof}

\begin{remark}
Lemma \ref{lem:SVMAOprop} implies that $\SVMAO$ belongs to the class
of set-valued maps called fans after \cite{Ioff81}, in particular
to the subclass of the compactly generated ones. Like convex processes
in the sense of Rockafellar (see \cite[Section 39]{Rock70}), they can be regarded as the simplest
(set-valued) generalization of linear operators. Unlike convex processes,
they typically fail to have convex graphs.
\end{remark}

The properties of $\pmfun$ (inherited from those of $\SVMAO$) needed
for the subsequent analysis are gathered in the next lemma.

\begin{lemma}[Properties of $\pmfun$]     \label{lem:pmfunprop}
Let the function $\pmfun:\R^n\longrightarrow [0,+\infty]$ be defined
as in (\ref{eq:pmfundef}). Under the standing assumptions $(a_0)-(a_2)$
the following assertions hold:
\begin{itemize}
  \item[(i)] $\dom\pmfun=\R^n$;

  \item[(ii)] $\pmfun$ is convex on $\R^n$;

  \item[(iii)] $\pmfun$ is Lipschitz continuous on $\R^n$.
\end{itemize}
If, in addition, $C$ and $Q$ are cones, then
\begin{itemize}
  \item[(iv)] $\pmfun$ is p.h..
\end{itemize}
\end{lemma}

\begin{proof}
Observe first that function $x\mapsto\dist{x}{C}$ is Lipschitz continuous on $\R^n$
and, as $C$ is convex, convex on $\R^n$. Thus, it suffices to prove that
all the above assertions hold true for the function $x\mapsto\exc{\SVMAO(x)}{Q}$.

(i) According to Lemma \ref{lem:SVMAOprop}(ii), $\SVMAO(x)$ is compact and,
as such, bounded, for every $x\in\R^n$. Thus, by the continuity of function
$y\mapsto\dist{y}{Q}$ one has $\exc{\SVMAO(x)}{Q}<+\infty$, for every $x\in\R^n$.

(ii) In the light of Lemma \ref{lem:SVMAOprop}(iv), it is true that
$$
  \SVMAO(tx_1+(1-t)x_2)\subseteq  t\SVMAO(x_1)+(1-t)\SVMAO(x_2),\quad
  \forall x_1,\ x_2\in\R^n, \forall  t\in [0,1].
$$
Then,  by virtue of the convexity of function $y\mapsto\dist{y}{Q}$,
it follows
\begin{eqnarray*}
  \exc{\SVMAO(tx_1+(1-t)x_2)}{Q} &=& \sup_{y\in\SVMAO(tx_1+(1-t)x_2)}
  \dist{y}{Q} \\
   &\le & \sup_{y\in  t\SVMAO(x_1)+(1-t)\SVMAO(x_2)}
   \dist{y}{Q} \\
   &\le& \sup_{y_1\in\SVMAO(x_1)\atop y_2\in\SVMAO(x_2)}
   \dist{ty_1+(1-t)y_2}{Q}  \\
   &\le& \sup_{y_1\in\SVMAO(x_1)\atop y_2\in\SVMAO(x_2)}
   [t\dist{y_1}{Q}+(1-t)\dist{y_2}{Q}]  \\
   &=& t\sup_{y_1\in\SVMAO(x_1)}\dist{y_1}{Q}+
   (1-t)\sup_{y_2\in\SVMAO(x_2)}\dist{y_2}{Q} \\
   &=& t\exc{\SVMAO(x_1)}{Q}+(1-t)\exc{\SVMAO(x_2)}{Q},  \\
   & & \qquad\qquad\qquad\qquad\forall x_1,\ x_2\in\R^n, \forall  t\in [0,1].
\end{eqnarray*}

(iii) By recalling that, if $A$ and $B$ are arbitrary subsets of $\R^m$, it is
$\exc{A}{Q}\le \exc{A}{B}+\exc{B}{Q}$, in the light of Lemma \ref{lem:SVMAOprop}(v)
one obtains
\begin{eqnarray*}
  \exc{\SVMAO(x_1)}{Q} &\le& \exc{\SVMAO(x_1)}{\SVMAO(x_2)}+ \exc{\SVMAO(x_2)}{Q} \\
   &\le & \beta\|x_1-x_2\|+\exc{\SVMAO(x_2)}{Q}, \quad\forall x_1,\ x_2\in\R^n,
\end{eqnarray*}
whence
$$
  \exc{\SVMAO(x_1)}{Q}-\exc{\SVMAO(x_2)}{Q}\le\beta\|x_1-x_2\|,
  \quad\forall x_1,\ x_2\in\R^n.
$$
By interchanging the role of $x_1$ and $x_2$, one can deduce the
Lipschitz continuity of $x\mapsto\exc{\SVMAO(x)}{Q}$.

(iv) In the light of Lemma \ref{lem:SVMAOprop}(iv), it is true that
$$
  \SVMAO(tx)=t\SVMAO(x),\quad
  \forall x\in\R^n,\ \forall  t>0.
$$
So it suffices to recall that, under the additional hypotheses, functions
$x\mapsto\dist{x}{C}$ and $y\mapsto\dist{y}{Q}$ are p.h. and hence
\begin{eqnarray*}
  \exc{\SVMAO(t x)}{Q} &=& \sup_{y\in\SVMAO(tx)}\dist{y}{Q}
  =\sup_{y\in t\SVMAO(x)}\dist{y}{Q} \\
  &=& \sup_{y\in \SVMAO(x)}\dist{ty}{Q}=
  t\sup_{y\in\SVMAO(x)}\dist{y}{Q} \\
  &=& t\exc{\SVMAO(x)}{Q},\quad\forall x\in\R^n,\ \forall  t>0.
\end{eqnarray*}
This completes the proof. \hfill $\square$
\end{proof}

\begin{remark}
(i) It is worth pointing out that assertion (iii) in Lemma \ref{lem:pmfunprop}
is not a mere consequence of assertion (ii). Indeed, the convexity of a function
whose domain is $\R^n$ can only entail a local Lipschitz behaviour for
such a function. Assertion (iii) says that $\pmfun$ is globally Lipschitz
continuous, with a constant $\beta$ (the same Lipschitz constant
as $\SVMAO$) valid all over $\R^n$.

(ii) All the properties of $\pmfun$ mentioned in Lemma \ref{lem:pmfunprop}
can be established in an alternative way, by exploiting well-known
properties of the support function and the representation of $\pmfun$
provided in Proposition \ref{pro:dualreppmfun}.
\end{remark}


\subsection{\bf Subtransversality of sets}

The set $C$ plays the role of a geometric constraint for a problem $(\RSFP)$.
A geometric qualification for collections of sets, which enables one to
treat separately the constraint $C$ and the set-valued inclusion problem
$\SVMAO(x)\subseteq Q$ relies on the concept of subtransversality. This concept,
which can be regarded as a metric generalization of the notion of transversality,
as it is classically understood in differential topology (see, for instance,
\cite[Chapter 3]{Hirs76}), formalizes a certain ``good" mutual arrangement
of several sets in space. It emerged as a regularity property under different
names in different topics within variational analysis, such as subdifferential
and normal calculus, Lipschitzian behaviour of set-valued maps (calmness and
metric subregularity), weak sharp minima (see \cite{BauBor96,Ioff17,KrLuTh17}).
Its development was definitely stimulated by investigations in the convergence
theory of projection methods for feasibility problems (see \cite{KrLuTh18}).
Recall that a pair of subsets $S_1,\, S_2$ of an Euclidean space is said to be
{\it subtransversal} at $\bar s\in S_1\cap S_2$ if there exist
positive constants $\eta$ and $\delta$ such that
$$
 \left[S_1+(\eta r)\Uball\right]\cap
 \left[S_2+(\eta r)\Uball\right]\cap\ball{\bar s}{\delta}
 \subseteq (S_1\cap S_2)+r\Uball,
 \quad\forall r\in [0,\delta).
$$
This property admits the following equivalent metric reformulation,
which will be exploited in the sequel: the pair  $S_1,\, S_2$ is
subtransversal at $\bar s$ iff there exist $\kappa,\, \delta>0$ such that
\begin{equation}    \label{in:defsubtransv}
  \dist{x}{S_1\cap S_2}\le\kappa\left[\dist{x}{S_1}+
  \dist{x}{S_2}\right],\quad\forall x\in\ball{\bar s}{\delta}
\end{equation}
(see \cite[Theorem 1(ii)]{KrLuTh18}).
The geometric idea behind the above metric inequality is
transparent: ``if you are close to both the sets of the pair,
then the intersection cannot be too far away" to quote \cite{BauBor96}.
Referring the reader to \cite{BauBor96,Ioff17,KrLuTh17,KrLuTh18}
for major details, for the purposes of the present analysis
it is useful to recall some situations,
in which this property takes place.
If it is $\bar s\in\inte(S_1\cap S_2)$, then the pair $S_1,\, S_2$
is subtransversal at $\bar s$.
If $S_1$ and $S_2$ are closed and convex and $\nullv\in\inte(S_1-S_2)$,
the pair $S_1,\, S_2$ is subtransversal at each point $\bar s
\in S_1\cap S_2$.
Whenever $S_1$ and $S_2$ are both polyhedral sets (namely, they can be
expressed as the intersection of finitely many closed halfspaces)
and $\bar s$ is any element in their intersection, the inequality in (\ref{in:defsubtransv})
becomes globally true ($\delta$ can be taken arbitrarily positive)
(see \cite[Theorem 8.35]{Ioff17}).

\vskip1cm


\section{Conditions for solvability and error bounds}
\label{Sect:3}

Before exposing the main result of the paper, it is worthwhile
to summarize general properties enjoyed by the solution set to $(\RSFP)$
(robust feasible solutions) in the adopted setting.

\begin{proposition}
Under the assumptions $(a_0)-(a_2)$,
\begin{itemize}
  \item[(i)] $\Solv$ is a closed convex (possibly empty) set;

  \item[(ii)] if, in addition, $C$ and $Q$ are cones, $\Solv$
  is also a cone (thereby, nonempty).
\end{itemize}
\end{proposition}

\begin{proof}
(i) This assertion is a straightforward consequence of Proposition \ref{pro:Solvchar}
and Lemma \ref{lem:pmfunprop}(ii) and (iii).

(ii) In this case, as $Q$ is a closed cone, $A_\omega^{-1}(Q)$ is a closed
cone for every $\omega\in\Omega$. Therefore, $\SVMAO^{+}(Q)=\bigcap_{\omega\in\Omega}
A_\omega^{-1}(Q)$ is a closed cone. Since so is $C$ too, one achieves
the thesis by remembering the characterization of $\Solv$ in (\ref{eq:Solvdef}).
\hfill $\square$
\end{proof}

Let us consider the partition $\{R_i\subseteq\R^n\ :\ i=1,2,3\}$ of the set $[\pmfun>0]$,
which is defined as follows
$$
  R_1=\{x\in\R^n\ :\  \exc{\SVMAO(x)}{Q}>0,\quad \dist{x}{C}>0\},
$$
$$
  R_2=\{x\in\R^n\ :\  \exc{\SVMAO(x)}{Q}=0,\quad \dist{x}{C}>0\},
$$
and
$$
  R_3=\{x\in\R^n\ :\  \exc{\SVMAO(x)}{Q}>0,\quad \dist{x}{C}=0\}.
$$
In what follows, it is convenient to collect in the clean-up set,
henceforth denoted by
$$
  \Omega_x=\{\omega\in\Omega\ :\ \dist{A_\omega x}{Q}=\exc{\SVMAO(x)}{Q}\},
$$
all those values of $\omega$ for which $A_\omega x$ is a farthest
element in $\SVMAO(x)$ from $Q$. Notice that, as $\SVMAO(x)$
is nonempty and compact, it is $\Omega_x\ne\varnothing$, for
every $x\in\R^n$. From the robustness point of view $\Omega_x$
singles out those uncertain scenarios (and corresponding matrices),
which are accountable for the worst violation of the constraint system.

In view of the statement of the next result, let us introduce the following constants
associated with the data of a given $(\RSFP)$:
\begin{eqnarray*}
  \ebcon_1=\inf\Biggl\{\|v\|\ :\  v\in\left[\clco\bigcup_{\omega\in\Omega_x}
  A_\omega^\top\frac{A_\omega x-\Pi_Q(A_\omega x)}{\dist{A_\omega x}{Q}}\right]+
  {x-\Pi_C(x)\over\dist{x}{C}}, \quad
     x\in R_1\Biggl\}
\end{eqnarray*}
\begin{eqnarray*}
  \ebcon_2=\inf\Biggl\{\|v\|\ :\ v\in\left[\clco\bigcup_{\omega\in\Omega}
  A_\omega^\top(\Ncone{A_\omega x}{Q}\cap\Uball)\right]+
  {x-\Pi_C(x)\over\dist{x}{C}}, \quad
   x\in R_2\Biggl\}
\end{eqnarray*}
and
\begin{eqnarray*}
  \ebcon_3=\inf\Biggl\{\|v\|\ :\ v\in\left[\clco\bigcup_{\omega\in\Omega_x}
  A_\omega^\top\frac{A_\omega x-\Pi_Q(A_\omega x)}{\dist{A_\omega x}{Q}}\right]+
  \left(\Ncone{x}{C}\cap\Uball\right),   \quad
  x\in R_3\Biggl\}.
\end{eqnarray*}

To understand the meaning of the above constants, it is helpful
to remark that
$$
   \frac{A_\omega x-\Pi_Q(A_\omega x)}{\dist{A_\omega x}{Q}}
   \in\Ncone{\Pi_Q(A_\omega x}{Q}\cap\Usfer,
   \quad\forall \omega\in\Omega_x,\ \forall x\in R_1\cup R_3,
$$
and
$$
  {x-\Pi_C(x)\over\dist{x}{C}}\in\Ncone{\Pi_C(x)}{C}
  \cap\Usfer,\quad\forall x\in R_1\cup R_2.
$$
On the basis of the above constructions, define
\begin{equation}    \label{eq:defebmincon}
   \ebconglo=\min\{\ebcon_1,\ebcon_2,\ebcon_3\}.
\end{equation}
One is now in a position to establish a condition for the existence
of robust feasible solutions and global error bounds for $(\RSFP)$, which are
expressed via the residual function $\pmfun$.

\begin{theorem}     \label{thm:psufcond}
With reference to $(\RSFP)$, let assumptions $(a_0)-(a_2)$ be satisfied.
If $\Omega$ is a separated compact topological space and
it is $\ebconglo>0$, then $\Solv\ne\varnothing$ and the following
estimate holds
\begin{equation}    \label{in:ebthm}
  \dist{x}{\Solv}\le\ebconglo^{-1}\pmfun(x),\quad\forall x\in\R^n.
\end{equation}
\end{theorem}

\begin{proof}
In the light of Proposition \ref{pro:Solvchar} the characterization
$\Solv=[\pmfun\le 0]$ holds.
Consequently, observe that in the case $[\pmfun>0]=\varnothing$
it happens that $\Solv=\R^n$, so the estimate in (\ref{in:ebthm})
becomes trivially true.
To the contrary, suppose now that $[\pmfun>0]\ne\varnothing$.
By Lemma \ref{lem:pmfunprop}(ii) $\pmfun:\R^n\longrightarrow\R$
is convex. Thus, all the assertions in the thesis can
be achieved by applying Proposition \ref{pro:sovebconvex}, upon
showing that $\ebconglo>0$ implies the condition in (\ref{def:ebsolvconvex})
to be actually satisfied, with $\tau=\ebconglo$. So, take an arbitrary $x\in [\pmfun>0]$.
If this is true, one (and only one) of the following cases must take place:

\noindent $\bullet$ Case $x\in R_1$:
according to the sum rule for subdifferential (\ref{eq:subdsumr}),
one has
\begin{equation}    \label{eq:subdexcdistsumr}
  \partial\pmfun (x)=\partial\exc{\SVMAO(\cdot)}{Q}(x)+
  \partial\dist{\cdot}{C}(x).
\end{equation}
By the Ioffe-Tikhomirov rule (\ref{eq:subdITr}), which can be invoked
under the assumption made on $\Omega$, one finds
\begin{eqnarray*}
   \partial\exc{\SVMAO(\cdot)}{Q}(x) &=& \partial\left(\sup_{\omega\in\Omega}
   \dist{A_\omega}{Q}\right)(x) \\
   &=& \clco\left(\bigcup_{\omega\in\Omega_x}\partial
   \dist{A_\omega}{Q}(x)\right).
\end{eqnarray*}
Then, the chain rule for the subdifferential calculus (\ref{eq:subdchainr})
enables one to write
$$
  \partial\dist{A(\omega,\cdot)}{Q}(x)=A_\omega^\top
  \left(\partial\dist{\cdot}{Q}(A_\omega x)\right).
$$
Since in the current case it is $A_\omega x\not\in Q$, because
$\omega\in\Omega_x$ and $\exc{\SVMAO(x)}{Q}>0$, then according to formula
(\ref{eq:subddist}) it results in
$$
   \partial\dist{\cdot}{Q}(A_\omega x)=\left\{
   {A_\omega x-\Pi_Q(A_\omega x)\over\dist{A_\omega x}{Q}}\right\}.
$$
Again, as it is $x\not\in C$, the same subdifferential rule gives
\begin{equation}    \label{eq:subddistCNconeB}
  \partial\dist{\cdot}{C}(x)=\left\{{x-\Pi_C(x)\over\dist{x}{C}}\right\}
\end{equation}
Thus, if $x\in R_1$ one obtains
$$
  \partial\pmfun(x)=
 \left[\clco\bigcup_{\omega\in\Omega_x}
  A_\omega^\top\frac{A_\omega x-\Pi_Q(A_\omega x)}{\dist{A_\omega x}{Q}}\right]+
  {x-\Pi_C(x)\over\dist{x}{C}}.
$$
By taking into account the definition of $\ebcon_1$, the last equality
allows one to say that if $v\in\partial\pmfun(x)$ it must be $\|v\|\ge\ebcon_1\ge\ebconglo>0$,
which shows that the condition in (\ref{def:ebsolvconvex}) turns out to be satisfied.

\noindent $\bullet$ Case $x\in R_2$:
notice that formula (\ref{eq:subdexcdistsumr}) still holds, whereas, as it is
$\exc{\SVMAO(\cdot)}{Q}(x)=0$, now it happens that $A_\omega x\in Q$, for every $\omega\in\Omega$.
In other words, it is $\Omega_x=\Omega$. In such an event,
by formula $(\ref{eq:subddist})$ one obtains
$$
  \partial\dist{\cdot}{Q}(A_\omega x)=\Ncone{A_\omega x}{Q}\cap\Uball,
$$
while equality (\ref{eq:subddistCNconeB}) is still valid.
As a consequence, if $x\in R_2$ one obtains
$$
  \partial\pmfun(x)=
 \left[\clco\bigcup_{\omega\in\Omega}
  A_\omega^\top\left(\Ncone{A_\omega x}{Q}\cap\Uball\right)\right]+
  {x-\Pi_C(x)\over\dist{x}{C}}.
$$
Thus, by recalling the definition of $\ebcon_2$, one deduces
that if $v\in\partial\pmfun(x)$ it must be  $\|v\|\ge\ebcon_2\ge\ebconglo>0$,
saying that the condition in (\ref{def:ebsolvconvex}) is again satisfied.

\noindent $\bullet$ Case $x\in R_3$: by proceeding in the same
manner as in the case $x\in R_1$, one obtains
$$
  \partial\exc{\SVMAO(\cdot)}{Q}(x)=\clco\bigcup_{\omega\in\Omega_x}
  A_\omega^\top\frac{A_\omega x-\Pi_Q(A_\omega x)}{\dist{A_\omega x}{Q}}.
$$
Since in the present case it is $x\in C$, by taking into account the
formula in (\ref{eq:subddist}) (second case), one finds
$$
   \partial\pmfun(x)=\left[\clco\bigcup_{\omega\in\Omega_x}
  A_\omega^\top\frac{A_\omega x-\Pi_Q(A_\omega x)}{\dist{A_\omega x}{Q}}\right]+
  \left(\Ncone{x}{C}\cap\Uball\right).
$$
Therefore, according to the definition of $\ebcon_3$, one can state
that if $v\in\partial\pmfun(x)$ it must be  $\|v\|\ge\ebcon_3\ge\ebconglo>0$.
Once again the condition in (\ref{def:ebsolvconvex}) turns out to be satisfied.

This completes the proof. \hfill $\square$
\end{proof}

The next example shows that the positivity condition on $\ebconglo$ in Theorem \ref{thm:psufcond}
is essential already for getting the existence of robust feasible solutions.

\begin{example}
Let $n=1$, $m=2$, and $(\RSFP)$ be defined by the data
$\Omega=[-1,1]\subseteq\R$, $C=[1,+\infty)$, $Q=[0,+\infty)\times\{0\}$,
with $\Amap:[-1,1]\longrightarrow{\bf M}_{2\times 1}(\R)$ given by
$$
  \Amap(\omega)=\binom{1}{\omega},
                      \qquad \omega\in [-1,1].
$$
Assumptions $(a_0)-(a_2)$ are evidently fulfilled.
The reformulation as a set-valued inclusion problem amounts to
$$
   \hbox{find $x\in [1,+\infty)$ such that}\quad
    \mathcal{A}_{\mathcal{U}}(x)\subseteq[0,+\infty)\times\{0\},
$$
where
$$
  \mathcal{A}_{\mathcal{U}}(x)=\conv\left\{\binom{x}{x},\binom{x}{-x}\right\}.
$$
It is readily seen that
$$
   \Solv=
   \mathcal{A}_{\mathcal{U}}^{+1}\left([0,+\infty)\times\{0\}\right)
   \cap [1,+\infty)=\{0\}\cap [1,+\infty)=\varnothing.
$$
With reference to the partition $\{R_i\subseteq\R\ :\ i=1,2,3\}$,
whenever taking $x\in (0,1)\subset R_1$, one finds $\Omega_x=\{-1,\, 1\}$.
In the case $\omega=1$, one has
$$
  \frac{A_1x-\Pi_{[0,+\infty)\times\{0\}}(A_1x)}{\dist{A_1x}{[0,+\infty)\times\{0\}}}
  =\frac{\displaystyle\binom{x}{x}-\displaystyle\binom{x}{0}}{x}=\binom{0}{1},
$$
whereas in the case $\omega=-1$, one has
$$
  \frac{A_{-1}x-\Pi_{[0,+\infty)\times\{0\}}(A_{-1}x)}{\dist{A_{-1}x}{[0,+\infty)\times\{0\}}}
  =\frac{\displaystyle\binom{x}{-x}-\displaystyle\binom{x}{0}}{x}=\binom{0}{-1},
$$
Consequently, one obtains
$$
  A_1^\top\frac{A_1x-\Pi_{[0,+\infty)\times\{0\}}(A_1x)}{\dist{A_1x}{[0,+\infty)\times\{0\}}}
  =\left(1\ \ 1\right)\binom{0}{1}=1
$$
and similarly
$$
  A_{-1}^\top\frac{A_{-1}x-\Pi_{[0,+\infty)\times\{0\}}(A_{-1}x)}{\dist{A_{-1}x}{[0,+\infty)\times\{0\}}}
  =(1\ -1)\binom{0}{-1}=1.
$$
Since for $x\in R_1$ it is
$$
  \frac{x-\Pi_{[1,+\infty)}(x)}{\dist{x}{[1,+\infty)}}=\frac{x-1}{1-x}=-1,
$$
it results in
\begin{eqnarray*}
  \clco\bigcup_{\omega\in\{-1,\, 1\}} A_\omega^\top
  \frac{A_\omega x-\Pi_{[0,+\infty)\times\{0\}}(A_\omega x)}{\dist{A_\omega x}{[0,+\infty)\times\{0\}}}
  +\frac{x-\Pi_{[1,+\infty)}(x)}{\dist{x}{[1,+\infty)}} \\
  =\{1\}+\{-1\}=\{0\}.
\end{eqnarray*}
Thus one sees that $\ebcon_1=\ebconglo=0$, that is the condition $\ebconglo>0$
fails to be satisfied.
\end{example}

As a further comment on Theorem \ref{thm:psufcond}, it must be remarked
that it provides a condition for solution existence and error bounds,
which is in general only sufficient, as illustrated by the next example.

\begin{example}
Let $n=1$, $m=2$, and $(\RSFP)$ be defined by the data
$\Omega=[0,1]\subseteq\R$, $C=[0,+\infty)=\R_+$, $Q=\R\times(-\infty,0]$,
with $\Amap:[0,1]\longrightarrow{\bf M}_{2\times 1}(\R)$ given by
$$
  \Amap(\omega)=\binom{1}{\omega},
                      \qquad \omega\in [0,1].
$$
The set-valued inclusion problem reformulation amounts to
$$
   \hbox{find $x\in [0,+\infty)$ such that}\quad
    \mathcal{A}_{\mathcal{U}}(x)\subseteq \R\times(-\infty,0],
$$
where
$$
  \mathcal{A}_{\mathcal{U}}(x)=\conv\left\{\binom{x}{0},\binom{x}{x}\right\}.
$$
For such an instance of $(\RSFP)$, as it is $\mathcal{A}_{[0,1]}^{+1}
\left(\R\times (-\infty,0]\right)=(-\infty,0]$, one finds
$$
  \Solv=(-\infty,0]\cap[0,+\infty)=\{0\}.
$$
In the present case, $[\pmfun>0]=R_2\cup R_3$, being $R_1=\varnothing$,
with  $R_2=(-\infty,0)$ and $R_3=(0,+\infty)$.
If $x\in R_2$, one has
$$
  \frac{x-\Pi_{\R_+}(x)}{\dist{x}{\R_+}}=-1,
$$
and
$$
  \Ncone{A_\omega x}{\R\times(-\infty,0]}=
  \left\{\begin{array}{ll}
           \{\nullv\} & \quad\hbox{ if } \omega\in (0,1], \\
           \R\times [0,+\infty) & \quad\hbox{ if } \omega=0.
         \end{array}  \right.
$$
Consequently, as it is
$$
  A_\omega^\top\left(\Ncone{A_\omega x}{\R\times(-\infty,0]}\right)
  =\left\{\begin{array}{ll}
           \{\nullv\} & \quad\hbox{ if } \omega\in (0,1], \\
           \{r\cos\theta\ :\ r\in [0,1],\ \theta\in [0,\pi]\} & \quad\hbox{ if } \omega=0,
         \end{array}  \right.
$$
one obtains
\begin{eqnarray*}
  \left[\clco\bigcup_{\omega\in [0,1]}
  A_\omega^\top(\Ncone{A_\omega x}{\R\times(-\infty,0]}\cap\Uball)\right]+
  {x-\Pi_{\R_+}(x)\over\dist{x}{{\R_+}}} &=&   \clco\{0,\, [-1,1]\}+\{-1\} \\
  &=& [-1,1]+\{-1\}  \\
  &=& [-2,0].
\end{eqnarray*}
It follows that $\ebconglo=\ebcon_2=0$, the positivity
condition on $\ebconglo$ thereby being violated.
On the other hand, since it is
$$
   \exc{\mathcal{A}_{\mathcal{U}}(x)}{\R\times(-\infty,0]}=\left\{\begin{array}{ll}
           0 & \hbox{ if } x\le 0, \\
           x & \hbox{ if } x>0
         \end{array}  \right\}=\max\{x,\, 0\}
$$
and $\dist{x}{\R_+}=-\min\{x,\, 0\}$, it is true that
$$
  \dist{x}{\Solv}=|x|\le \max\{x,\, 0\}-\min\{x,\, 0\},
  \quad\forall x\in\R.
$$
So a global error bound for the current instance of $(\RSFP)$ takes
place, with $\tau=1$, even in the violation of the condition on $\ebconglo$.
\end{example}

\begin{example}     \label{ex:RSFPthm1ok}
Let $n=m=2$ and consider the instance of $(\RSFP)$ defined by the data
$\Omega=[0,1]\subseteq\R$, $C=\R^2$, $Q=-\R^2_+=\R^2_-$, with $\Amap:[0,1]
\longrightarrow{\bf M}_{2\times 2}(\R)$ being given by
$$
  \Amap(\omega)=\left(\begin{array}{cc}
                        \omega\quad & 1-\omega \\
                        1-\omega\quad & \omega \\
                      \end{array}    \right),
                      \qquad \omega\in [0,1].
$$
A robust solution to the corresponding $(\RSFP)$ is thereby
any $x\in\R^2$ satisfying the inequality system
\begin{equation}     \label{in:svipinsys}
  \left\{\begin{array}{ll}
                        \omega x_1 + (1-\omega)x_2 & \le 0 \\
                        (1-\omega)x_1+\omega x_2  & \le 0 \\
                      \end{array}    \right. \qquad \forall\omega\in [0,1].
\end{equation}
In the present case $\mathcal{U}=\Amap([0,1])$ turns out to be the Birkhoff polytope
$\Bpoly_2$ of all the doubly stochastic matrices in ${\bf M}_{2\times 2}(\R)$,
whose extreme elements are the two permutation matrices
$$
  \left(\begin{array}{cc}
                        1\quad & 0 \\
                        0\quad & 1 \\
                      \end{array}  \right)
  \qquad\hbox{ and }\qquad
   \left(\begin{array}{cc}
                        0\quad & 1 \\
                        1\quad & 0 \\
                      \end{array}    \right)
$$
(see, for instance, \cite[Birkhoff-von Neumann Theorem]{Barv02}).
According to the Krein-Milman representation of $\Bpoly_2$, one
can write
$$
  \Amap(\omega)=\omega\left(\begin{array}{cc}
                        1\quad & 0 \\
                        0\quad & 1 \\
                      \end{array}  \right)+
              (1-\omega)\left(\begin{array}{cc}
                        0\quad & 1 \\
                        1\quad & 0 \\
                      \end{array}    \right),
                       \qquad \omega\in [0,1]
$$
and hence it results in
\begin{equation}    \label{eq:aomegaxrep}
   A_\omega x=\left[\omega\left(\begin{array}{cc}
                        1\quad & 0 \\
                        0\quad & 1 \\
                      \end{array}  \right)+
              (1-\omega)\left(\begin{array}{cc}
                        0\quad & 1 \\
                        1\quad & 0 \\
                      \end{array}  \right)\right]
                      \binom{x_1}{x_2}
                      =\omega\binom{x_1}{x_2}+(1-\omega)\binom{x_2}{x_1}.
\end{equation}
Notice that $\binom{x_2}{x_1}$ is the vector symmetric to $x=\binom{x_1}{x_2}$ with
respect to the symmetry ax $x_2=x_1$. Thus, it is plain to see that
\begin{equation}   \label{eq:SVAO01rep}
  \mathcal{A}_{\mathcal{U}}(x)=\left\{ \omega\binom{x_1}{x_2}+(1-\omega)\binom{x_2}{x_1}\ :\
  \omega\in [0,1]\right\}=\conv\left\{\binom{x_1}{x_2},\, \binom{x_2}{x_1}\right\}.
\end{equation}
Clearly, $\mathcal{A}_{\mathcal{U}}$ is single-valued iff its argument
is $x=t(1,1)$, with $t\in\R$.
The solution set $\Solv$ of the problem in consideration coincides
with the solution set of the set-valued inclusion problem
$$
   \hbox{find $x\in\R^2$ such that}\quad
    \mathcal{A}_{\mathcal{U}}(x)\subseteq\R^2_-.
$$
By recalling the representations in (\ref{eq:aomegaxrep}) and (\ref{eq:SVAO01rep})
as well as their plain geometrical meaning (otherwise, by elementary
reasoning on the system (\ref{in:svipinsys})), it is readily seen that
$$
  \Solv=\R^2_-.
$$
Assumption $(a_0)-(a_2)$ are clearly satisfied by the problem data
(in particular, as for $(a_2)$ remember Example \ref{ex:Bpoly}),
while $\Omega$, equipped with the topological structure induced by
$\R$, is separated and compact. So, in order to apply Theorem \ref{thm:psufcond},
it remains to check that condition $\ebconglo>0$ is valid. To do so,
observe first that, since for this problem it is $R_1=R_2=\varnothing$,
then one has $\ebcon_1=\ebcon_2=+\infty$ and therefore $\ebconglo=
\ebcon_3$. To estimate this constant, it is useful to observe that,
since it is
$$
   A_\omega x\in\conv\left\{\binom{x_1}{x_2},\, \binom{x_2}{x_1}\right\},
   \quad\forall \omega\in [0,1],\ \forall x\in\R^2,
$$
then for every $x\in\R^2\backslash\R^2_-$ one obtains
$$
  a_\omega x=\frac{A_\omega x-\Pi_{\R^2_-}(A_\omega x)}{\dist{A_\omega x}{\R^2_-}}
  =\left\{\begin{array}{cc}
            \binom{1}{0} & \hbox{ if } (A_\omega x)_1>0,\ (A_\omega x)_2\le 0\\ \\
              \frac{A_\omega x}{\|A_\omega x\|}& \hbox{ if } (A_\omega x)_1>0,\ (A_\omega x)_2>0\\ \\
            \binom{0}{1} & \hbox{ if } (A_\omega x)_1\le 0,\ (A_\omega x)_2>0,
          \end{array}
  \right.
$$
where it is $A_\omega x=((A_\omega x)_1,\, (A_\omega x)_2)$. It follows
\begin{equation}    \label{eq:aomegaxset}
  \{a_\omega x\in\R^2\ :\ \omega\in [0,1],\ x\in\R^2\backslash\R^2_-\}
  =\Usfer\cap\R^2_+.
\end{equation}
On account of the representation in (\ref{eq:SVAO01rep}), one can
establish that for every $x\in\R^2\backslash\R^2_-$  it holds
$$
  \exc{\mathcal{A}_{\mathcal{U}}(x)}{\R^2_-}=\left\{
  \begin{array}{ll}
    \|x\|, & \quad\hbox{ if } x_1>0,\ x_2>0 \\  \\
    \max\{x_1,\, x_2\} & \quad\hbox{ if } x_1>0,\ x_2\le 0
     \ \hbox{ or }\    x_1\le 0,\ x_2>0,
  \end{array}  \right.
$$
while it is
\begin{eqnarray*}
  \Omega_x &=& \left\{\tilde\omega\in [0,1]\ :\ \dist{A(\tilde\omega,x)}{\R^2_-}=
  \max_{\omega\in [0,1]}\dist{A_\omega x}{\R^2_-}\right\}   \\
  &=& \left\{
  \begin{array}{ll}
    \{0,\, 1\}  & \quad\hbox{ if } x\in\R^2\backslash\R^2_-,\ x_1\ne x_2 \\  \\
    \hbox{[0,1]} & \quad\hbox{ if } x\in\R^2\backslash\R^2_-,\ x_1=x_2.
  \end{array}  \right.
\end{eqnarray*}
The reader should notice that, since it is $\dist{\cdot\,}{\R^2}\equiv 0$,
then
\begin{equation}   \label{eq:pmfunexcSVMAO}
   \pmfun(x)=\exc{\mathcal{A}_{\mathcal{U}}(\cdot)}{\R^2_-},\quad\forall
   x\in\R^2\backslash\R^2_-,
\end{equation}
so $\pmfun$ is actually p.h. and convex, consistently with what established
in Lemma \ref{lem:pmfunprop} (iii) and (iv).

What has been observed about $\Omega_x$ implies that, in computing $\ebcon_3$, one needs
to consider only the extreme elements of $\Amap([0,1])=\Bpoly_2$,
namely
$$
  A_0^\top=A_0=\left(\begin{array}{cc}
                        0\quad & 1 \\
                        1\quad & 0 \\
                      \end{array}    \right)
  \quad\hbox{ and }\quad
   A_1^\top=A_1=\left(\begin{array}{cc}
                        1\quad & 0 \\
                        0\quad & 1 \\
                      \end{array}  \right).
$$
By symmetry of the set $\Usfer\cap\R^2_+$ with respect to
the ax $x_1=x_2$, one sees
$$
  A_0^\top\left(\Usfer\cap\R^2_+\right)=\Usfer\cap\R^2_+,
$$
whereas it is evident that
$$
  A_1^\top\left(\Usfer\cap\R^2_+\right)=\Usfer\cap\R^2_+.
$$
From (\ref{eq:aomegaxset}) it follows that for every $x\in\R^2\backslash\R^2_-$
it holds
$$
  \bigcup_{\omega\in\Omega_x}A_\omega^\top a_\omega x=
  \left\{\binom{a_1(\omega,x)}{a_2(\omega,x)},\, \binom{a_2(\omega,x)}{a_1(\omega,x)}
  \right\},
$$
where $A_\omega x=(a_1(\omega,x),\, a_2(\omega,x))$, and hence
$$
  \bigcup_{x\in\R^2\backslash\R^2_-}\left[\clco
  \bigcup_{\omega\in\Omega_x}A_\omega^\top a_\omega x\right] =
  \clco\left(\Usfer\cap\R^2_+\right)
  =\left\{v=(v_1,v_2)\in\Uball\ :\ v_1+v_2\ge 1\right\}.
$$
Since it is
$$
  \Ncone{x}{\R^2}\cap\Uball=\{\nullv\},
$$
from the above considerations it is possible to achieve
the estimate
$$
  \ebcon_3=\inf\{\|v\|\ :\ v=(v_1,v_2)\in\Uball\ :\ v_1+v_2\ge 1 \}=
  \left\|\binom{1/2}{1/2}\right\|=\frac{1}{\sqrt{2}}>0.
$$
This shows that hypothesis $\ebconglo>0$ is valid for the problem
under study, thereby enabling to apply Theorem \ref{thm:psufcond}.
It has been already seen that $\Solv\ne\varnothing$. Let us check the
validity of the estimate in (\ref{in:ebthm}). It is readily seen that
\begin{eqnarray*}
   \dist{x}{\Solv} &=& \dist{x}{\R^2_-} \\
   &=& \left\{\begin{array}{ll}
    \|x\|, & \quad\hbox{ if } x_1>0,\ x_2>0, \\  \\
    \max\{x_1,\, x_2\} & \quad\hbox{ if } x_1>0,\ x_2\le 0
     \ \hbox{ or }\    x_1\le 0,\ x_2>0,
  \end{array}  \right. \\
  &=& \exc{\mathcal{A}_{\mathcal{U}}(x)}{\R^2_-},\quad\forall x\in\R^2\backslash\R^2_-.
\end{eqnarray*}
By remembering the equality in (\ref{eq:pmfunexcSVMAO}), it is possible
to conclude that for the problem under study the estimate
$$
  \dist{x}{\Solv}\le\sqrt{2}\pmfun(x),\quad\forall x\in\R^2
$$
is actually valid.
\end{example}

In order to deepen the understanding of the condition provided by Theorem \ref{thm:psufcond}
it may be helpful to  consider what error bound condition can be derived from it
in the special case $\card\Omega=1$, as to say in the absence of
uncertainty, with $(\RSFP)$ thereby collapsing to $(\SFP)$.
This is done in the special case in which $C$ and $Q$ are both closed convex cones.
In such setting, clearly it is $\Solv\ne\varnothing$.
A fact which is employed in the proof of the next result is
pointed out in the remark below.

\begin{remark}     \label{rem:0capsumcones}
Let $K_1$ and $K_2$ be two cones such that $K_1\cap (-K_2)=\{\nullv\}$.
Then, for every $r>0$ it must be
$$
   \nullv\notin K_1+\left(K_2\backslash\inte r\Uball\right).
$$
Indeed, if it were $\nullv\in K_1+\left(K_2\backslash\inte r_0\Uball\right)$
for some $r_0>0$, then there would exist $k_1\in K_1$ and $k_2\in
K_2\backslash\{\nullv\}$ such that $\nullv=k_1+k_2$. This would imply
$k_1=-k_2\ne\nullv$ and hence $K_1\cap (-K_2)\ne\{\nullv\}$.
\end{remark}

\begin{theorem}     \label{thm:autrobcase}
With reference to $(\SFP)$, suppose that:
\begin{itemize}
\item[(i)] $C\subseteq\R^n$ and $Q\subseteq\R^m$ are (nonempty)
closed, convex cones;

\item[(ii)] $\krn A^\top\cap\dcone{Q}=\{\nullv\}$;

\item[(iii)] $A^\top\left(\dcone{Q}\right)\cap\left(-\dcone{C}\right)=\{\nullv\}$.
\end{itemize}
Then there exists $\tau>0$ such that the following estimate holds
\begin{equation}  \label{le:erboSolvnrob}
  \dist{x}{\Solv}\le\tau[\dist{Ax}{Q}+\dist{x}{C}],
  \quad\forall x\in\R^n.
\end{equation}
\end{theorem}

\begin{proof}
Observe that, on account of Remark \ref{rem:0capsumcones}, if taking the two cones
$K_1=A^\top\left(\dcone{Q}\right)$ and $K_2=\dcone{C}$, hypothesis (iii)
implies
$$
  \nullv\notin A^\top\left(\dcone{Q}\right)+\left[\dcone{C}\backslash
  \inte\Uball\right],
$$
which, in turn, passing to subsets, entails
\begin{equation}   \label{nin:ATdQdC1}
  \nullv\notin A^\top\left(\dcone{Q}\cap\Uball\right)+\left[\dcone{C}
  \cap\Usfer\right].
\end{equation}
Hypothesis (iii) is also equivalent to
$$
  \dcone{C}\cap\left[-A^\top\left(\dcone{Q}\right)\right]=\{\nullv\}.
$$
Thus, if taking now $K_1=\dcone{C}$ and $K_2=A^\top\left(\dcone{Q}\right)$,
it implies
$$
  \nullv\notin \left[A^\top\left(\dcone{Q}\right)\backslash \inte r\Uball\right]
  +\dcone{C},\quad\forall r>0,
$$
which, again passing to subsets, entails
\begin{equation}  \label{nin:ATdQdCr}
  \nullv\notin \left[A^\top\left(\dcone{Q}\right)\backslash\inte r\Uball\right]
  +\left(\dcone{C}\cap\Uball\right),  \quad\forall r>0.
\end{equation}
Notice that owing to hypothesis (ii) it is $\nullv\notin A^\top\left(\dcone{Q}
\cap\Usfer\right)$, otherwise it would exist $q\in\dcone{Q}$, with $q\ne\nullv$,
such that $A^\top q=\nullv$, so that $q\in\krn A^\top\cap\dcone{Q}\ne\{\nullv\}$.
Therefore, since $A^\top\left(\dcone{Q}\cap\Usfer\right)$ is compact, then for
some $r_0>0$ it must be
$$
  A^\top\left(\dcone{Q}\cap\Usfer\right)\subseteq A^\top\left(\dcone{Q}\right)\backslash\inte
  r_0\Uball.
$$
By combining the last inclusion with the exclusion in (\ref{nin:ATdQdCr}),
one obtains
\begin{equation}   \label{nin:ATdQdC2}
  \nullv\notin A^\top\left(\dcone{Q}\cap\Usfer\right)+\left[\dcone{C}
  \cap\Uball\right].
\end{equation}
Notice that both the sets
$$
  A^\top\left(\dcone{Q}\cap\Uball\right)+\left[\dcone{C}
  \cap\Usfer\right] \quad\hbox{ and }\quad
  A^\top\left(\dcone{Q}\cap\Usfer\right)+\left[\dcone{C}
  \cap\Uball\right]
$$
are compact, as a sum of two compact sets. As a result, their union
$$
  D=\left\{A^\top\left(\dcone{Q}\cap\Uball\right)+\left[\dcone{C}
  \cap\Usfer\right]\right\}\bigcup
  \left\{A^\top\left(\dcone{Q}\cap\Usfer\right)+\left[\dcone{C}
  \cap\Uball\right]\right\}
$$
is still compact and, by virtue of the exclusion in (\ref{nin:ATdQdC1}) and
(\ref{nin:ATdQdC2}), one finds that $\nullv\notin D$.  Consequently,
it must be
\begin{equation}  \label{in:tauinfDpos}
  \tau=\inf\{\|v\|\ :\ v\in D\}>0.
\end{equation}
Now, on the basis of the above considerations, let us show that,
in the particular case in which $\card\Omega=1$, the condition $\ebconglo>0$
is actually satisfied.
In the present event, it is $\SVMAO(x)=\{Ax\}$ and hence $\exc{\SVMAO(x)}{Q}=
\dist{Ax}{Q}$. Consequently, the partition $\{R_i\subseteq\R^n\ :\ i=1,2,3\}$ of $[\pmfun>0]$,
takes the form
$$
  R_1=\{x\in\R^n\ :\  \dist{Ax}{Q}>0,\quad \dist{x}{C}>0\},
$$
$$
  R_2=\{x\in\R^n\ :\  \dist{Ax}{Q}=0,\quad \dist{x}{C}>0\},
$$
and
$$
  R_3=\{x\in\R^n\ :\  \dist{Ax}{Q}>0,\quad \dist{x}{C}=0\}.
$$
Observe that if it is $x\in R_1$, since it is $\Ncone{\Pi_Q(Ax)}{Q}
\subseteq\dcone{Q}$, then
$$
  A^\top\frac{Ax-\Pi_Q(Ax)}{\dist{Ax}{Q}}+
  {x-\Pi_C(x)\over\dist{x}{C}}\in
  A^\top\left(\dcone{Q}\cap\Uball\right)+\left[\dcone{C}
  \cap\Usfer\right];
$$
if it is $x\in R_2$, then
$$
  A^\top\left(\Ncone{Ax}{Q}\cap\Uball\right)+
  {x-\Pi_C(x)\over\dist{x}{C}}\subseteq
  A^\top\left(\dcone{Q}\cap\Uball\right)+\left[\dcone{C}
  \cap\Usfer\right];
$$
if it is $x\in R_3$, then
$$
  A^\top\frac{Ax-\Pi_Q(Ax)}{\dist{Ax}{Q}}+
 \left(\Ncone{x}{C}\cap\Uball\right)\subseteq
  A^\top\left(\dcone{Q}\cap\Usfer\right)+\left[\dcone{C}
  \cap\Uball\right].
$$
By remembering the definition of $\ebcon_1$, $\ebcon_2$, $\ebcon_3$,
and $\ebconglo$, from the above inclusions it is possible to deduce
$$
  \ebconglo\ge\tau>0.
$$
Notice furthermore that, by hypothesis (i) and $\card\Omega=1$,
assumptions $(a_0)-(a_2)$ are trivially fulfilled.
Thus the thesis follows from Theorem \ref{thm:psufcond}.
\hfill $\square$
\end{proof}

\begin{remark}
The reader should notice that hypothesis (ii) of Theorem \ref{thm:autrobcase}
is fulfilled, in particular, whenever $A$ is onto ($n\ge m$ and $A$
with full rank), inasmuch as in this case one has
$\krn(A^\top)=(A(\R^n))^\perp=\{\nullv\}$.

Hypothesis (iii) of Theorem \ref{thm:autrobcase} is fulfilled,
in particular, whenever $(\SFP)$ is unconstrained, i.e. $C=\R^n$.
\end{remark}

\vskip1cm


\section{Error bounds in the polyhedral case}   \label{Sect:4}

In order for complementing the result exposed in Section \ref{Sect:3},
the present section focuses on global error bound conditions
in the special case in which the data defining a $(\RSFP)$ happen
to be polyhedral. To do so, throughout the current section, assumption
$(a_0)$ is enhanced as follows:
\begin{itemize}
  \item[$(a_0^+)$] $\varnothing\ne C\subseteq\R^n$ is closed and convex, and
   $\{\nullv\}\ne Q\subsetneqq\R^m$ is a closed, convex, pointed cone.
\end{itemize}

The following concept, introduced in \cite{Uder19}, turned out to be
effective in addressing solvability and stability issues within the
context of the set-valued inclusion problems.

\begin{definition}[{\bf $Q$-increase property}]    \label{def:Cincrpropdef}
Let $\mathcal{F}:\R^n\rightrightarrows\R^m$ be a set-valued
map, let $\{\nullv\}\ne Q\subsetneqq\R^m$ be a closed, convex cone, and let
$x_0\in\dom\mathcal{F}$. The map $\mathcal{F}$ is said to be {\it (metrically) $Q$-increasing} at
$x_0$ if there exist $\alpha>1$ and $\delta>0$ such that
\begin{equation}    \label{def:Cincrprop}
   \forall r\in (0,\delta]\quad\ \exists z\in\ball{x_0}{r}\ :\
   \ball{\mathcal{F}(z)}{\alpha r}\subseteq \ball{\mathcal{F}(x_0)+Q}{r}.
\end{equation}
The value
\begin{equation}    \label{def:Cincrbxbo}
   \inc{\mathcal{F}}{Q}{x_0}=\sup\{\alpha>1:\ \exists\delta>0 \hbox{ for which (\ref{def:Cincrprop})
   holds\,}\}
\end{equation}
is called {\it exact bound of $Q$-increase} of $\mathcal{F}$ at $x_0$. If $\mathcal{F}$
is $Q$-increasing at each $x_0\in\R^n$, then it is said to be $Q$-increasing on $\R^n$.
\end{definition}

Examples of entire classes of set-valued maps enjoying the above property
can be found in \cite{Uder19,Uder24}, along with further comments
about connections with the decrease principle in variational analysis.

\begin{remark}   \label{rem:unexCincrprop}
In the sequel, the fact will be exploited that, whenever
a set-valued map $\mathcal{F}$ is $Q$-increasing at $x_0$ while it is
$\mathcal{F}(x_0)\not\subseteq Q$, then, if $\alpha$ and $\delta$ are as in Definition
\ref{def:Cincrpropdef}, for every $r\in (0,\delta]$ the inclusion in
(\ref{def:Cincrprop}) must be necessarily satisfied by $z\in\ball{x_0}{r}
\backslash\{x_0\}$ (see \cite[Remark 4]{Uder24}).
\end{remark}

Let us introduce below a qualification condition on which the current approach
to the study of error bounds relies:
with reference to the set-valued inclusion problem (\ref{eq:defSVMAO}),
the map $\SVMAO:\R^n\rightrightarrows\R^m$
is said to satisfy a {\it Slater-like qualification condition} if
$$
  \inte Q\ne\varnothing \quad\hbox{ and }\quad
  \exists u\in\Uball\ :\ \SVMAO(u)\subseteq\inte Q.
  \leqno (\SCQ)
$$
The next lemma states that the qualification condition $(\SCQ)$ provides
a sufficient condition for the map $\SVMAO$ to be $Q$-increasing on $\R^n$
and, at the same time, an uniform estimate from below of its $Q$-increase
exact bound.

\begin{lemma}    \label{lem:SCQincr}
Let $\SVMAO$ be defined as in (\ref{eq:defSVMAO}). Under the assumptions
$(a_0^+)-(a_2)$, if (\SCQ) holds then $\SVMAO$ is $Q$-increasing on $\R^n$
and there exists $\alpha_{\SVMAO}>1$ such that
$$
   \inc{\SVMAO}{Q}{x}\ge\alpha_{\SVMAO},\quad\forall x\in\R^n.
$$
\end{lemma}

\begin{proof}
Let $u\in\R^n$ be as in $(\SCQ)$. According to Lemma \ref{lem:SVMAOprop}(ii)
the set $\SVMAO(u)$ is compact. As a consequence, there exists $\eta>0$ such that
\begin{equation}     \label{in:AUetaB}
  \SVMAO(u)+\eta\Uball\subseteq Q.
\end{equation}
Take an arbitrary $x\in\R^n$ and $r>0$ and set $z=x+ru$. Then, it is $z\in\ball{x}{r}$
and, by the $Q$-concavity and the positive homogeneity property of
$\SVMAO$ ensured by Lemma \ref{lem:SVMAOprop}(iv), on the account of the
inclusion in (\ref{in:AUetaB})
one finds
\begin{eqnarray*}
  \ball{\SVMAO(z)}{(\eta+1)r} &=& \SVMAO(x+ru)+(\eta+1)r\Uball
  \subseteq \SVMAO(x)+r\SVMAO(u)+\eta r\Uball+r\Uball \\
   &=& \SVMAO(x)+r\left(\SVMAO(u)+\eta\Uball\right)+r\Uball
   \subseteq \SVMAO(x)+rQ+r\Uball \\
   &=& \ball{\SVMAO(x)+Q}{r}.
\end{eqnarray*}
Thus, by setting $\alpha_{\SVMAO}=\eta+1>1$, the above inclusion shows
that $\SVMAO$ fulfils the condition in (\ref{def:Cincrprop}), with
$\alpha=\alpha_{\SVMAO}$ and any $\delta>0$. By recalling the definition in
(\ref{def:Cincrbxbo}), one can conclude that $\inc{\SVMAO}{Q}{x}\ge\alpha_{\SVMAO}$,
thereby completing the proof.
\hfill $\square$
\end{proof}

The role of the qualification condition $(\SCQ)$ is explained
by the next proposition. For its proof, it is convenient to recall
the following basic properties of the excess operator: let $Q\subseteq\R^m$
be a closed, convex cone; then
\begin{itemize}
  \item[$(p_1)$] $\forall S\subseteq\R^m$, $\exc{S+Q}{Q}=\exc{S}{Q}$;

  \vskip.5cm

  \item[$(p_2)$] $\forall r>0$ and $\forall S\subseteq\R^m$, such that
  it is $\exc{S}{Q}>0$, it holds $\exc{\ball{S}{r}}{Q}=\exc{S}{Q}+r$
\end{itemize}
(for their proofs, the reader is referred to \cite[Remark 2.1(iv)]{Uder19}
and \cite[Lemma 2.2]{Uder19}, respectively).

\begin{proposition}     \label{pro:ebpolySCQ}
With reference to $(\RSFP)$, under the assumptions
$(a_0^+)-(a_2)$ suppose that:
\begin{itemize}
  \item[(i)]  $C$ and $\SVMAO^{+1}(Q)$ are polyhedral;

  \item[(ii)] the qualification condition (\SCQ) is satisfied.
\end{itemize}
Then, if $\Solv\ne\varnothing$, there exists $\tau>0$ such that the
following estimate holds
$$
  \dist{x}{\Solv}\le \tau\pmfun(x),\quad\forall x\in\R^n.
$$
\end{proposition}

\begin{proof}
By virtue of Lemma \ref{lem:SCQincr}, hypothesis (ii) guarantees that
$\SVMAO$ is $Q$-increasing on $\R^n$. Let $\alpha\in (1,\alpha_{\SVMAO})$, where
$\alpha_{\SVMAO}$ as in  Lemma \ref{lem:SCQincr}.
Let us show that this property yields the following implication, which is
a sort of Caristi-like property, as is meant in the context of variational analysis (see,
for instance, \cite{ArZhZh19} and \cite[Chapter I.2]{GraDug03}):
\begin{eqnarray}  \label{in:Caristiprop}
  &\forall x_0\not\in \SVMAO^{+}(Q)\quad\exists \hat{x}\in\R^n
 \backslash\{x_0\}\ :\ \nonumber \\
  &\exc{\SVMAO(\hat{x})}{Q}+(\alpha-1)\|\hat{x}-x_0\|
  \le\exc{\SVMAO(x_0)}{Q}.
\end{eqnarray}
Indeed, notice that by continuity of function $x\mapsto\exc{\SVMAO(x)}{Q}$
(remember the argument in the proof of Lemma \ref{lem:pmfunprop}(iii))
and the fact that $\exc{\SVMAO}{x_0}>0$, there exists $\delta>0$ such that
$$
  \exc{\SVMAO}{x}>0,\quad\forall x\in\ball{x_0}{\delta}.
$$
By the $Q$-increase property of $\SVMAO$ at $x_0$, if taking $r\in (0,\delta)$,
there exists $z\in\ball{x_0}{r}$, such that
\begin{equation}   \label{in:Qincrx_0z}
  \ball{\SVMAO(z)}{\alpha r}\subseteq \ball{\SVMAO(x_0)+Q}{r}.
\end{equation}
It is to be noticed that, since $\SVMAO(x_0)\nsubseteq Q$, then, according to
what was observed in Remark \ref{rem:unexCincrprop}, it must be $z\in\R^n\backslash\{x_0\}$.
By exploiting the properties $(p_1)$ and $(p_2)$ of the excess operator
and the inclusion in (\ref{in:Qincrx_0z}),
which is valid because $z\in\ball{x_0}{\delta}$ and hence $\exc{\SVMAO(z)}{Q}>0$,
one obtains
\begin{eqnarray*}
  \exc{\SVMAO(z)}{Q} &=& \exc{\ball{\SVMAO(z)}{\alpha r}}{Q}-\alpha r \\
   &\le & \exc{\ball{\SVMAO(x_0)+Q}{r}}{Q}-\alpha r \\
   &=& \exc{\SVMAO(x_0)+Q}{Q}+r-\alpha r \\
   &=& \exc{\SVMAO(x_0)}{Q}-(\alpha-1)r.
\end{eqnarray*}
By taking into account that $\|z-x_0\|\le r$, from the above inequality it
follows
$$
   \exc{\SVMAO(z)}{Q}+(\alpha-1)\|z-x_0\|\le \exc{\SVMAO(x_0)}{Q},
$$
which shows the validity of (\ref{in:Caristiprop}).
Such a Caristi-like condition, by virtue of the Bishop-Phelps/Ekeland
variational principle (see, for instance, \cite[Chapter I.2]{GraDug03}), implies
the existence of $x_\alpha\in\R^n$ such that
\begin{equation}    \label{eq:xalphainSolv}
  \exc{\SVMAO(x_\alpha)}{Q}=0
\end{equation}
and
\begin{equation}    \label{in:distxalphax0}
  \|x_\alpha-x_0\|\le \frac{\exc{\SVMAO(x_0)}{Q}}{\alpha-1}.
\end{equation}
By combining (\ref{eq:xalphainSolv}) and (\ref{in:distxalphax0}),
one obtains
$$
  \dist{x_0}{\SVMAO^{+1}(Q)}\le\|x_\alpha-x_0\|\le
   \frac{\exc{\SVMAO(x_0)}{Q}}{\alpha-1}.
$$
By the arbitrariness of $x_0$, the above argument enables
to achieve the following estimate
\begin{equation}    \label{in:ebSVMAO+1Q}
    \dist{x}{\SVMAO^{+1}(Q)}\le
   \frac{\exc{\SVMAO(x)}{Q}}{\alpha-1},\quad\forall x\in\R^n,
\end{equation}
in the case $x\in\SVMAO^{+1}(Q)$ the above inequality being
trivially true.
Now, by virtue of hypothesis (i), the metric reformulation of
the global subtransversality of pairs of polyhedral sets given
in (\ref{in:defsubtransv}) allows one to write for a proper $\kappa>0$
\begin{eqnarray*}
  \dist{x}{\Solv} &=& \dist{x}{C\cap\SVMAO^{+1}(Q)}   \\
  &\le &\kappa \left[\dist{x}{C}+\dist{x}{\SVMAO^{+1}(Q)}\right],
  \quad\forall x\in\R^n.
\end{eqnarray*}
Thus, by taking into account the estimate in (\ref{in:ebSVMAO+1Q}),
it results in
$$
 \dist{x}{\Solv}\le\kappa\left[\dist{x}{C}+
 \frac{\exc{\SVMAO(x)}{Q}}{\alpha-1}\right]\le\tau
 \pmfun(x),\quad\forall x\in\R^n,
$$
provided that one takes $\tau\ge\kappa\max\{1,\, (\alpha-1)^{-1}\}$.
This completes the proof.
\hfill $\square$
\end{proof}

The hypotheses of Proposition \ref{pro:ebpolySCQ} are formulated only
partially in terms of problem data. In order to derive from it a more
satisfactory condition for error bounds, one can refer to a quantitative
form of regularity for linear maps. To this aim, recall that a linear map,
represented by $A\in\Lin$, is said to be covering (at a linear rate)
if there exists $\sigma>0$ such that $A\Uball\supseteq\sigma\Uball$,
in which case the quantity
$$
  \sur{A}=\sup\{\sigma>0\ :\ A\Uball\supseteq\sigma\Uball\}
$$
is called exact covering bound of $A$. Roughly speaking, it provides
a measure of how much $A$ is surjective, with $\sur{A}=0$ signaling
the failure of such a property. Historically, this particular manifestation of
the regularity of a map played a crucial role in understanding the deep connections
between different topics of variational analysis (see \cite{Ioff17,Mord06}).
For the purposes of the present analysis, it suffices to recall that
the following characterization of the exact covering bound holds:
\begin{equation}     \label{eq:surcharcov}
  \sur{A}=\dist{\nullv}{A^\top\Usfer}=
  \min_{u\in\Usfer}\|A^\top u\|
\end{equation}
(see \cite[Corollary 1.58]{Mord06}).

\begin{theorem}[Error bound under polyhedral assumptions]    \label{thm:ebpolysufcond}
With reference to $(\RSFP)$, under the assumptions
$(a_0^+)-(a_2)$, suppose that:
\begin{itemize}
  \item[(i)]  $C$, $Q$ and $\Amap(\Omega)=\mathcal{U}$ are polyhedral;

  \item[(ii)] $\sur{\SVMAO}=\displaystyle\inf_{\omega\in\Omega}\sur{A_\omega}>0$;

  \item[(iii)] $\inte\left(\displaystyle\bigcap_{\omega\in\Omega}A_\omega^{-1}(Q)\right)
                \ne\varnothing$.
\end{itemize}
Then, if $\Solv\ne\varnothing$, there exists $\tau>0$ such that the
following estimate holds
\begin{equation}    \label{in:ebpolypro}
  \dist{x}{\Solv}\le \tau\pmfun(x),\quad\forall x\in\R^n.
\end{equation}
\end{theorem}

\begin{proof}
According to hypothesis (i), as a polytope, by the Krein-Milman theorem
$\mathcal{U}$ can be expressed as
$$
   \mathcal{U}=\clco\{A_1,\dots,A_k\},
$$
where $k\in\N\backslash\{0\}$, by a proper choice of $A_i\in\mathcal{U}$.
Therefore, one readily sees that
\begin{equation}     \label{eq:interi:kAi}
  \SVMAO^{+1}(Q)=\bigcap_{i=1}^k A_i^{-1}(Q).
\end{equation}
Notice that, since $Q$ is polyhedral and hence each set $A_i^{-1}(Q)$ is polyhedral as well,
for $i=1,\dots,k$, so is the intersection in (\ref{eq:interi:kAi}).
This fact makes the hypothesis (i) of Proposition \ref{pro:ebpolySCQ}
satisfied.

By hypothesis (iii), there exist $u\in\R^n$ and $\epsilon>0$ such that
$$
  u+\epsilon\Uball\subseteq\bigcap_{\omega\in\Omega}A_\omega^{-1}(Q).
$$
Since $Q$ is pointed, it is possible to assume that $u\ne\nullv$.
Indeed, otherwise, if it is $\epsilon\Uball\subseteq\bigcap_{\omega\in\Omega}
A_\omega^{-1}(Q)$, then for some $v\in\R^n\backslash\{\nullv\}$
it must be
$$
  A_\omega v\in Q \qquad\hbox{ and }\qquad A_\omega v\in -Q,
  \qquad\forall\omega\in\Omega,
$$
whence $A_\omega v=\nullv\in Q$, for every $\omega\in\Omega$.
It follows that $\nullv\ne v\in\SVMAO^{+1}(Q)$ and
$$
 v+\epsilon\Uball\subseteq\SVMAO^{+1}(Q)+\SVMAO^{+1}(Q)=\SVMAO^{+1}(Q).
$$
Furthermore, since $\SVMAO^{+1}(Q)$ is a cone, it is possible to assume
that $u\in\Uball$.
According to hypothesis (ii), if $\sigma\in (0,\sur{\SVMAO})$, by linearity
of any $A_\omega$ one has
$$
  A_\omega(\epsilon\Uball)\supseteq \sigma\epsilon\Uball,\quad\forall
  \omega\in\Omega,
$$
which implies
$$
   A_\omega u+\sigma\epsilon\Uball\subseteq A_\omega(u+\epsilon\Uball)
   \subseteq Q,\quad\forall\omega\in\Omega.
$$
The last inclusion says that
$$
   \SVMAO(u)+\sigma\epsilon\Uball\subseteq Q,
$$
meaning that the qualification condition $(\SCQ)$ is satisfied.
All its hypotheses being satisfied, Proposition \ref{pro:ebpolySCQ}
can be invoked to achieve the assertion in the thesis.
\hfill $\square$
\end{proof}

\begin{example}
If considering the problem in Example \ref{ex:RSFPthm1ok}, which falls in
the polyhedral case, it happens that,
while as seen Theorem \ref{thm:psufcond} does apply, Theorem \ref{thm:ebpolysufcond} can
not be applied, because hypothesis (ii) fails to be satisfied. Indeed,
it is readily seen that $\sur{\Amap(1/2)}=0$ and, consequently, $\sur{\SVMAO}=0$.
This reveals that the conditions proposed in Theorem \ref{thm:ebpolysufcond}
are only sufficient for the validity of error bounds.
It is also worth observing that, for the problem at the issue, it holds
$$
  \inte\left(\bigcap_{\omega\in [0,1]}A_\omega^{-1}(\R^2_-)\right)=
  \inte\left(A_0^{-1}(\R^2_-)\cap A_1^{-1}(\R^2_-)\right)=
  \inte\R^2_-\ne\varnothing,
$$
so hypothesis (iii) of Theorem \ref{thm:ebpolysufcond} turns out to be
satisfied.
\end{example}

\begin{example}
Let $n=m=2$ and let a $(\RSFP)$ be defined by the data
$\Omega=[0,1]\subseteq\R$, $C=\R^2_+$, $Q=\R^2_-$, with $\Amap:[0,1]
\longrightarrow{\bf M}_{2\times 2}(\R)$ being given by
$$
  \Amap(\omega)=\left(\begin{array}{cc}
                        1\quad & -\omega \\
                        0\quad & 1 \\
                      \end{array}    \right)
                      =\omega A_0+(1-\omega)A_1,
                      \quad \omega\in [0,1].
$$
where
$$
   A_0=\left(\begin{array}{rr}
                        1\quad & -1 \\
                        0\quad & 1 \\
                      \end{array}  \right) \qquad \hbox{ and }\qquad
                    A_1=\left(\begin{array}{rr}
                        1\quad & 0 \\
                        0\quad & 1 \\
                      \end{array}  \right).
$$
In other words, the Krein-Milman representation of $\mathcal{U}$
is given by $\mathcal{U}=\clco\{A_0,\, A_1\}$.
Thus, the introduced $(\RSFP)$ falls in the polyhedral case,
thereby fulfilling hypothesis (i) of Theorem \ref{thm:ebpolysufcond}.
Since it is
$$
  A_0^{-1}(\R^2_-)=\{x\in\R^2\ :\ x_1\le x_2\le 0\}
  \qquad\hbox{and}\qquad
   A_1^{-1}(\R^2_-)=\R^2_-,
$$
it results in
$$
   \SVMAO^{+1}(\R^2_-)=\{x\in\R^2\ :\ x_1\le x_2\le 0\}.
$$
Consequently, one finds
$$
  \Solv=\R^2_+\cap\SVMAO^{+1}(\R^2_-)=\{\nullv\}.
$$
One readily sees that assumptions $(a_0^+)-(a_2)$ are satisfied by
the problem at hand. Since it is
\begin{eqnarray*}
  \inte\left(\bigcap_{\omega\in [0,1]}A_\omega^{-1}(\R^2_-)\right) &=&
  \inte\left(A_0^{-1}(\R^2_-)\cap A_1^{-1}(\R^2_-)\right)=
  \{x\in\R^2\ :\ x_1< x_2< 0\}   \\
  &\ne& \varnothing,
\end{eqnarray*}
hypothesis (iii) of Theorem is \ref{thm:ebpolysufcond}
fulfilled.
As for hypothesis (ii), observe that, according to (\ref{eq:surcharcov}),
it is
\begin{eqnarray*}
  \sur{A_\omega} &=& \min_{u\in\Usfer}\|A_\omega^\top\|=
  \min_{u=(u_1,u_2)\in\Usfer}\ \left\|\binom{u_1}{-\omega u_1+u_2}\right\| \\
   &=&\min_{\theta\in [0,2\pi]} \sqrt{\cos^2\theta+(-\omega\cos\theta+\sin\theta)^2}  \\
   &=& \min_{\theta\in [0,2\pi]}  \sqrt{1+\omega^2\cos^2\theta-2\omega\sin\theta\cos\theta}.
\end{eqnarray*}
By minimizing the continuous function $(\omega,\theta)\mapsto\psi(\omega,\theta)$,
defined by
$$
  \psi(\omega,\theta)=\omega^2\cos^2\theta-2\omega\sin\theta\cos\theta,
$$
over the compact box $[0,1]\times[0,2\pi]$, one finds
$$
  \min_{(\omega,\theta)\in[0,1]\times[0,2\pi]}
  \psi(\omega,\theta)=\psi(1,\theta_*)\ge -\frac{3}{4},
$$
where $\theta_*\in [\pi/4,\pi/3]$. It follows
$$
  \sur{\SVMAO}=\inf_{\omega\in [0,1]}\sur{A_\omega}\ge
  \sqrt{1+\psi(1,\theta_*)}\ge\frac{1}{2}>0,
$$
thereby showing that hypothesis (ii) is satisfied. Thus, Theorem
\ref{thm:ebpolysufcond} can be applied to the present $(\RSFP)$.
Let us check the validity of the global error bounds prescribed by
this theorem.
It is clear that
$$
  \dist{x}{\Solv}=\dist{x}{\nullv}=\|x\|,\quad\forall
  x\in\R^2.
$$
On the other hand, one has
$$
   \dist{x}{\R^2_-} =\left\{\begin{array}{ll}
    \|x\|, & \quad\hbox{ if } x\in\R^2_-, \\  \\
    0,  & \quad\hbox{ if } x\in\R^2_+,\\  \\
    -\min\{x_1,\, x_2\} & \quad\hbox{ if } x_1<0,\ x_2>0
     \ \hbox{ or }\    x_1>0,\ x_2<0,
  \end{array}  \right.
$$
and
$$
   \dist{x}{\SVMAO^{+1}(\R^2_-)} =\left\{\begin{array}{ll}
    \|x\|, & \quad\hbox{ if } x\in\R^2_+, \hbox{ or } x_1>0,\ x_2\ge -x_1  \\  \\
    x_2,  & \quad\hbox{ if } x_1<0,\ x_2>0,  \\  \\
    \displaystyle\frac{|x_1-x_2|}{\sqrt{2}}, & \quad\hbox{ if } x_2<|x_1|,   \\   \\
    0,  &  \quad\hbox{ if } x_1\le x_2\le 0.
  \end{array}  \right.
$$
Consequently, by summing up the above functions one obtains
$$
  \dist{x}{\R^2_-}+\dist{x}{\SVMAO^{+1}(\R^2_-)} =\left\{\begin{array}{lll}
     \|x\|, & \hbox{ if } & x\in\R^2_+,   \\   \\
     -\min\{x_1,\, x_2\}+\|x\|  & \hbox{ if }  & x_2<0, \\
                                            &  & x_2\ge -x_1,  \\  \\
     -\min\{x_1,\, x_2\}+\displaystyle\frac{|x_1-x_2|}{\sqrt{2}}, & \hbox{ if } & x_1\ge 0,  \\
                                                                          & &\ x_2\le -x_1,  \\   \\
     \|x\|+\displaystyle\frac{|x_1-x_2|}{\sqrt{2}}, & \hbox{ if } & x_1\le 0,  \\
                                                                     & &\ x_2\le x_1,  \\   \\
     \|x\|,  & \hbox{ if } & x_1\le x_2\le 0,    \\   \\
     -\min\{x_1,\, x_2\}+x_2,  & \hbox{ if }  &  x_1<0,  \\
                                                                          & &\ x_2>0.  \\   \\
  \end{array}  \right.
$$
Since it holds
$$
   \|x\|\le\sqrt{2}\max\{|x_1|,\, |x_2|\},\quad\forall x\in\R^2,
$$
it is possible to conclude that, for the problem under discussion, the global
error bound estimate in (\ref{in:ebpolypro}) holds true with $\tau\ge\sqrt{2}$.
\end{example}

\vskip1cm


\section{Conclusions}

The contents of this paper describe a methodology for dealing
with split feasibility problems, when some of their data happen
to be affected by uncertainty. This methodology follows a robust
approach to convex optimization problems under uncertainty.
Evidences are brought of the fact that the recently developed
set-valued inclusion theory offers a comfortable framework,
where to analyse solution existence and error bound issues by
techniques of convex and variational analysis.
The condition established for solvability of the robust counterpart
of a split feasibility problem and for global error bounds for the
associated solution set are expressed in terms of metric projections
and classic constructions from convex analysis, on the basis of the problem
data.
The fact that the main result is only a sufficient condition for error bounds
leaves space for further investigations on the issue.
Another direction for future research work deals with the case of $(\RSFP)$,
for which also the data $C$ and $Q$ are affected by uncertainty.
A different development of the proposed analysis should address a
more general class of robust split feasibility problems, where
the assumption of boundedness on $\mathcal{U}$ in $(a_2)$ is removed.
Instead, a complete rethinking of the approach should be required if assumption
$(a_1)$ is dropped out, leading to nonlinear split feasibility problems.
Nonetheless, even in such a more general setting, the author is
confident that an approach via set-valued inclusions could afford
useful insights into the question.

\vskip1cm



%
%



\end{document}